\documentclass[12pt]{article}
\usepackage{amssymb,amsthm,amsmath,amsfonts,booktabs,latexsym,tikz,hyperref,cleveref,ytableau}
\usepackage[hmargin=1in,vmargin=1in]{geometry}

\newtheorem{thm}{Theorem}[section]
\newtheorem{prop}[thm]{Proposition}
\newtheorem{cor}[thm]{Corollary}
\newtheorem{lem}[thm]{Lemma}
\newtheorem{conj}[thm]{Conjecture}

\theoremstyle{definition}
\newtheorem{defn}[thm]{Definition}
\newtheorem{exa}[thm]{Example}

\newcommand{\da}{\hs{-2pt}\downarrow}
\DeclareMathOperator{\Rk}{Rk}
\DeclareMathOperator{\R}{R}
\DeclareMathOperator{\SR}{SR}
\DeclareMathOperator{\RB}{RB}
\DeclareMathOperator{\SRB}{SRB}
\DeclareMathOperator{\Hilb}{Hilb}
\newcommand{\cSA}{\cS\cA}
\newcommand{\cSAB}{\cS\cA\cB}

\newcommand{\ben}{\begin{enumerate}}
\newcommand{\een}{\end{enumerate}}
\newcommand{\ble}{\begin{lem}}
\newcommand{\ele}{\end{lem}}
\newcommand{\bth}{\begin{thm}}
\renewcommand{\eth}{\end{thm}}
\newcommand{\bpr}{\begin{prop}}
\newcommand{\epr}{\end{prop}}
\newcommand{\bco}{\begin{cor}}
\newcommand{\eco}{\end{cor}}
\newcommand{\bcon}{\begin{conj}}
\newcommand{\econ}{\end{conj}}
\newcommand{\bde}{\begin{defn}}
\newcommand{\ede}{\end{defn}}
\newcommand{\bex}{\begin{exa}}
\newcommand{\eex}{\end{exa}}
\newcommand{\barr}{\begin{array}}
\newcommand{\earr}{\end{array}}
\newcommand{\btab}{\begin{tabular}}
\newcommand{\etab}{\end{tabular}}
\newcommand{\beq}{\begin{equation}}
\newcommand{\eeq}{\end{equation}}
\newcommand{\bea}{\begin{eqnarray*}}
\newcommand{\eea}{\end{eqnarray*}}
\newcommand{\bal}{\begin{align*}}
\newcommand{\bce}{\begin{center}}
\newcommand{\ece}{\end{center}}
\newcommand{\bpi}{\begin{picture}}
\newcommand{\epi}{\end{picture}}
\newcommand{\bpp}{\begin{picture}}
\newcommand{\epp}{\end{picture}}
\newcommand{\bfi}{\begin{figure} \begin{center}}
\newcommand{\efi}{\end{center} \end{figure}}
\newcommand{\bprf}{\begin{proof}}
\newcommand{\eprf}{\end{proof}\medskip}

\newcommand{\bsl}{\begin{slide}{}}
\newcommand{\esl}{\end{slide}}
\newcommand{\bfr}{\begin{frame}}
\newcommand{\efr}{\end{frame}}

\newcommand{\comp}{\models}

\newcommand{\hqed}{\hfill \qed}

\newcommand{\eqqed}[1]{$\rule{1ex}{0ex}\hfill{\dil#1}\hfill\qed$}

\newcommand{\ol}{\overline}

\newcommand{\hs}[1]{\hspace{#1}}
\newcommand{\hso}[1]{\hspace{-1pt}}
\newcommand{\vs}[1]{\vspace{#1}}
\newcommand{\qmq}[1]{\quad\mbox{#1}\quad}

\newcommand{\sbe}{\subseteq}

\newcommand{\setm}{\setminus}

\newcommand{\iso}{\cong}
\newcommand{\Cong}{\equiv}

\newcommand{\zh}{\hat{0}}
\newcommand{\oh}{\hat{1}}

\newcommand{\ptn}{\vdash}

\newcommand{\case}[4]{\left\{\barr{ll}#1&\mbox{#2}\\#3&\mbox{#4}\earr\right.}
\newcommand{\fl}[1]{\lfloor #1 \rfloor}

\newcommand{\gau}[2]{\left[ \barr{c} #1 \\ #2 \earr \right]}

\def\<{\langle}
\def\>{\rangle}

\newcommand{\spn}[1]{\langle{#1}\rangle}
\newcommand{\ree}[1]{(\ref{#1})}

\newcommand{\ra}{\rightarrow}

\newcommand{\al}{\alpha}
\newcommand{\be}{\beta}

\newcommand{\de}{\delta}
\newcommand{\ep}{\epsilon}

\newcommand{\om}{\omega}

\newcommand{\si}{\sigma}
\renewcommand{\th}{\theta}

\newcommand{\Si}{\Sigma}

\newcommand{\bx}{{\bf x}}

\newcommand{\bbN}{{\mathbb N}}

\newcommand{\bbQ}{{\mathbb Q}}
\newcommand{\bbR}{{\mathbb R}}

\newcommand{\bbZ}{{\mathbb Z}}
\newcommand{\cA}{{\cal A}}

\newcommand{\cB}{{\cal B}}

\newcommand{\cL}{{\cal L}}

\newcommand{\cS}{{\cal S}}
\newcommand{\cT}{{\cal T}}

\newcommand{\fS}{{\mathfrak S}}

\newcommand{\Sb}{\ol{S}}

\DeclareMathOperator{\Des}{Des}

\DeclareMathOperator{\inv}{inv}
\DeclareMathOperator{\Inv}{Inv}

\DeclareMathOperator{\maj}{maj}

\DeclareMathOperator{\Mod}{mod}

\DeclareMathOperator{\rk}{rk}

\DeclareMathOperator{\sgn}{sgn}

\newcommand{\dil}{\displaystyle}

\begin{document}

\pagestyle{plain}

\title{$q$-Stirling numbers in type $B$ 
}
\author{
Bruce E. Sagan\\[-5pt]
\small Department of Mathematics, Michigan State University,\\[-5pt]
\small East Lansing, MI 48824, USA, {\tt sagan@math.msu.edu}
\\
and
\\
Joshua P. Swanson\\[-5pt]
\small Department of Mathematics, Univesity of Southern California,\\[-5pt]
\small Los Angeles, CA 90007, USA, {\tt swansonj@usc.edu}
}

\date{\today\\[10pt]
	\begin{flushleft}
	                                       \\[5pt]
	\end{flushleft}}

\maketitle

\begin{abstract}
Stirling numbers, which count partitions of a set and permutations in the symmetric group, have found extensive application in combinatorics, geometry, and algebra.  We study analogues and $q$-analogues of these numbers corresponding to the Coxeter group of type $B$.  In particular, we show how they are related to complete homogeneous and elementary symmetric polynomials; demonstrate how they $q$-count signed partitions and permutations; compute their ordinary, exponential, and $q$-exponential generating functions; and prove various identities about them.  Ordered analogues of the $q$-Stirling numbers of the second kind have recently appeared in conjectures of Zabrocki and of Swanson--Wallach concerning the Hilbert series of certain super coinvariant algebras.  We provide conjectural bases for these algebras and show that they have the correct Hilbert series. 
\end{abstract}

\tableofcontents

\section{Introduction}
\label{I}

Let $\bbZ$ and $\bbN$ be the integers and nonnegative integers, respectively. If $n\in\bbN$. then we will be interested in two intervals of integers, namely
$$
[n]=\{1,2,\ldots,n\} \qmq{and}
\spn{n} = \{-n,-n+1,\ldots,n-1,n\}.
$$
If $n\in\bbZ$ and $q$ is a variable then we let
\begin{align*}
[n]_q &= \frac{1 - q^n}{1 - q},\\  
[n]_q!& = [n]_q [n-1]_q \cdots [1]_q,\\
[n]_q!! &= [n]_q [n-2]_q [n-4]_q \cdots,
\end{align*}
 where the double factorial ends at $[1]_q$ or $[2]_q$ depending on whether $n$ is odd or even, respectively.
 Note that if $n<0$ then $[n]_q!=[n]_q!!=1$ because both are the empty product.
 Note also that if $n\ge0$ then 
 $$
 [n]_q = 1+q+q^2+\cdots + q^{n-1}.
 $$
We will often drop the subscript $q$ in such notation if there can be no confusion about whether $[n]$ denotes the polynomial or the set.

\subsection{Classical $q$-Stirling numbers}

The purpose of the present work is to give a comprehensive treatment of the Stirling numbers in type $B$ and their $q$-analogues. We begin by defining the various classical Stirling numbers in type $A$ and their $q$-analogues recursively. Each recursion in this paper will always have $n \in \bbN$ and $k \in \bbZ$ with the same boundary condition when $n=0$. We start with the Stirling numbers of the second kind, which arise more frequently in the context of $q$-analogues.

\begin{defn}
The {\em (type $A$) Stirling numbers of the second kind} are $S(n,k)$ for $n\in\bbN$ and $k\in\bbZ$ 
defined by the initial condition $S(0,k)=\de_{0,k}$  (Kronecker delta) and for $n \geq 1$,
\begin{equation}\label{eq:S_rec}
S(n,k) = S(n-1,k-1)+k S(n-1,k).
\end{equation}

\end{defn}

It is well known that $S(n,k)$ is the number of partitions of the set $[n]$ into $k$ non-empty subsets called {\em blocks}.  These partitions are in bijection with subspaces of dimension $n-k$ in the intersection lattice of the type $A_{n-1}$ Coxeter group.

We will be primarily interested in $q$-analogues of Stirling numbers. There are in fact {\em two} standard $q$-analogues of the Stirling numbers of the second kind, which differ by a $q$-shift. We will call the {\em (type $A$) $q$-Stirling numbers of the second kind} the polynomials $S[n, k]$ in the variable $q$ obtained by replacing \eqref{eq:S_rec} with
\begin{equation}\label{eq:Sq_rec}
S[n,k] = S[n-1,k-1]+[k] S[n-1,k].
\end{equation}
Sometimes $\Sb[n, k] = q^{\binom{k}{2}} S[n, k]$ is also encountered, which replaces \eqref{eq:S_rec} with
\begin{equation}\label{eq:Sq_rec2}
\Sb[n,k] = q^{k-1} \Sb[n-1,k-1]+[k] \Sb[n-1,k].
\end{equation}

Some of the work related to $q$-Stirling numbers concerns the following ordered analogue.
\begin{defn}
The {\em (type $A$) ordered $q$-Stirling numbers of the second kind} are
$$
S^o[n,k]= [k]! S[n,k],
$$
and $\Sb^o[n,k] = [k]! \Sb[n, k]$.
\end{defn}
\noindent When $q=1$, $S^o(n, k)$ counts the number of \textit{ordered set partitions} of $[n]$ with $k$ blocks.

We now recall the Stirling numbers of the first kind and their $q$-analogue.
\begin{defn}
The {\em (signless, type $A$) Stirling numbers of the first kind} are $c(n,k)$ for $n \in \bbN$ and $k \in \bbZ$ defined by the initial condition $c(0, k) = \delta_{0, k}$ and for $n \geq 1$,
\begin{equation}\label{eq:c_rec}
c(n,k)=c(n-1,k-1)+(n-1)c(n-1,k).
\end{equation}
Their signed counterparts are
$$
s(n,k)=(-1)^{n-k} c(n,k).
$$
\end{defn}

Combinatorially, $c(n,k)$ counts the number of elements in the symmetric group  of permutations of $[n]$ which have $k$ cycles in their disjoint cycle decomposition. The {\em (signless, type $A$) $q$-Stirling numbers of the first kind} are the polynomials $c[n, k]$ obtained by replacing \eqref{eq:c_rec} with
\begin{equation}\label{eq:cq_rec}
c[n,k]=c[n-1,k-1]+[n-1] c[n-1,k].
\end{equation}
Their signed counterparts are $s[n, k] = (-1)^{n-k} c[n, k]$.

Stirling numbers of both kinds have been extensively studied in combinatorics and have interesting applications in algebra and geometry.  See the texts of Sagan~\cite{sag:aoc} or Stanley~\cite{sta:ec1} for more information. We next review the literature on $q$-Stirling numbers.

\subsection{Existing $q$-Stirling literature}

The $q$-analogues above have been frequently studied, especially those of the second kind.  In the history which follows, we will sometimes make no distinction between the two different variants of the $S[n,k]$.  The second kind $q$-Stirlings first appeared in the work of Carlitz on abelian fields~\cite{car:af} and $q$-Bernoulli polynomials~\cite{car:qbn}.
Then Gould~\cite{gou:qsn} studied $q$-Stirling numbers of the first and second kinds, defining them in terms of elementary and complete homogeneous symmetric polynomials.  
The $q$-Stirling numbers of the first kind also appeared in the work of Gessel~\cite{ges:qae} on a $q$-analogue of the exponential formula.

By weighting the blocks of a partition, Garsia~\cite{gar:miq} was the first to show that $\Sb[n,k]$  can be considered as the generating function for a statistic on set partitions, and Rawlings~\cite{raw:rmi} generalized this approach.
Milne~\cite{mil:rgf} showed that restricted growth functions (which are equinumerous with set partitions) could be used to give two statistics, one for each version of $S[n,k]$, both of which are similar to the inversion statistic on permutations.
Wachs and White~\cite{WW:pqs} built on Milne's work, giving two more inversion-like statistics on restricted growth functions with the same distribution.  There is also an analogue of the major index for these Stirling numbers as shown by Sagan~\cite{sag:mss}.

Ehrenborg and Readdy~\cite{ER:jaq} interpreted the $\Sb[n,k]$ in terms of juggling sequences and then Ehrenborg~\cite{ehr:diq} used this interpretation to evaluate various determinants whose entries are these polynomials.
Other work on $q$-Stirling  numbers of the second kind has been done by Garsia and Remmel~\cite{GR:qcr} and by Leroux~\cite{ler:rmq}.

In fact, in Carlitz's original paper~\cite{car:af}, the sum which arises in the context of Stirling numbers is actually for $\Sb^o[n,k]$ and he has to divide by $2^{\binom{k}{2}}[k]!$ to get the quantity in which he is interested.  
Zeng and Zhang~\cite{ZZ:qan}  used analytic means to prove a formula relating the $\Sb^o[n,k]$ and a $q$-analogue of the $q$-exponential numbers.
A connection between the ordered polynomials and ordered set partitions was obtained by Steingr\'{\i}msson in a 2001 preprint which was finally published in 2020~\cite{ste:sop}.   There he gave eight statistics analogous to the ones of Wachs and White.  He also made a number of conjectures and posed an open problem (to find a combinatorial proof of Zeng and Zhang's identity) which spurred a number of other authors to work on ordered set partitions~\cite{IKZ:cgf,IKZ:ems,KZ:ems,KZ:nsp,RW:eme,Wil:eme}.  See
Ishikawa, Kasraoui, and Zeng~\cite{IKZ:emss} for a survey.

Haglund, Rhoades, and Shimozono~\cite{HRS:osp} showed that there is a connection between ordered set partitions, generalized coinvariant algebras, and the Delta Conjecture (see the recent proof of the rise version \cite{DAM:delta}). In related work, Zabrocki~\cite{zab:mdc} conjectured that
the tri-graded Hilbert series of the type $A$ superdiagonal
coinvariant algebra has coefficients which are the  $S^o[n,k]$.
Swanson and Wallach~\cite{SW:hdf} made a corresponding conjecture in type $B$. This led them to conjecture that an alternating sum involving these ordered Stirling numbers equals one. We prove this as \Cref{alt:sumA}.

\subsection{Stirling numbers in type $B$ and their $q$-analogues}

The Stirling numbers in type $B$ have appeared sporadically in the literature over the last several decades. They can be defined as follows.

\begin{defn}
The {\em type $B$ Stirling numbers of the second kind} are defined by replacing \eqref{eq:S_rec} with
\begin{equation}\label{eq:SB_rec}
S_B(n,k)= S_B(n-1,k-1)+ (2k+1) S_B(n-1,k).
\end{equation}
The ordered version of $S_B(n,k)$ is
$$
S_B^o(n,k) = (2k)!! S_B(n,k).
$$
The {\em (signless) type $B$ Stirling numbers of the first kind} are defined by replacing \eqref{eq:c_rec} with
\begin{equation}\label{eq:cB_rec}
c_B(n,k) = c_B(n-1,k-1) + (2n-1) c_B(n-1,k),
\end{equation}
with the signed version being 
$$
s_B(n,k)=(-1)^{n-k} c_B(n,k).
$$
\end{defn}
The reason for calling these ``type $B$'' is because the $S_B(n,k)$ and $s_B(n, k)$ are the Whitney numbers of the second and first kind for the intersection lattice $\cL_{B_n}$ of the hyperplane arrangement of the Coxeter group $B_n$, as follows from the work of Zaslavsky~\cite{zas:grs}. See also \Cref{lsp}.

The $S_B(n,k)$ appear implicitly in the work of Dowling~\cite{dow:cgl} on certain lattices and, as already mentioned, in that of Zaslavsky~\cite{zas:grs} concerning signed graphs.  They were defined explicitly by Dolgachev and Lunts when studying representations of Weyl groups~\cite{DL:cfr} and Reiner~\cite{rei:ncp} in his work on noncrossing partitions for classical reflection groups.  An analogue of the $s_B(n,k)$ appears implicitly in a formula of Shephard and Todd~\cite{ST:fur} for the characteristic polynomial of the intersection lattice of an arbitrary finite complex reflection group. Neither they, nor their $q$-analogue defined below, have been explicitly defined elsewhere to our knowledge. See \Cref{ci} for combinatorial interpretations of $S_B(n, k)$ and $c_B(n, k)$ involving signed set partitions and signed permutations.

Our primary interest lies in the following $q$-analogues of the preceding type $B$ Stirling numbers.

\begin{defn}
The {\em type $B$ $q$-Stirling numbers of the second kind} are defined by replacing \eqref{eq:S_rec} with
\begin{equation}\label{eq:SBq_rec}
S_B[n,k]= S_B[n-1,k-1]+ [2k+1] S_B[n-1,k].
\end{equation}
The ordered version of $S_B[n,k]$ is
$$
S_B^o[n,k] = [2k]!! S_B[n,k].
$$
The {\em (signless) type $B$ Stirling numbers of the first kind} are defined by replacing \eqref{eq:c_rec} with
\begin{equation}\label{eq:cBq_rec}
c_B[n,k] = c_B[n-1,k-1] + [2n-1] c_B[n-1,k],
\end{equation}
with the signed version being 
$$
s_B[n,k]=(-1)^{n-k} c_B[n,k].
$$
\end{defn}

As far as we know, the $S_B[n,k]$ have only appeared once before (in their ordered form) in the previously cited paper of Swanson and Wallach.  We will have more to say about this in Sections~\ref{oai} and~\ref{ca}.
As we were preparing this article, we became aware that Bagno, Garber, and Komatsu~\cite{BGK:qrS} were also studying type $B$ analogues of the Stirling numbers.  Their primary tool is the use of  signed analogues of restricted growth functions (RGFs) as opposed to working  with signed partitions and permutations as we do here.  Interestingly, their type $B$ analogue of the $s(n,k)$ differs significantly. 
They also use a statistic on RGFs for $S_B[n,k]$ which, when translated into the language of signed partitions, is different from ours although the two can be related by a simple bijection.
They also consider a  variant where the elements of the interval $[r]$ appear in different blocks or cycles for some fixed $r$. 
 The overlapping results between our two papers are \Cref{cS:eh} (b) and (d), as well as \Cref{egf:thm} (a). 

\subsection{Summary of results}

Here we summarize some of our main results. See the following sections for more details, references, and additional results. In \Cref{sp}, we show

\begin{itemize}
    \item $S_B[n, k] = h_{n-k}([1], [3], \ldots, [2k+1])$ and
    $$t^n = \sum_{k=0}^n S_B[n, k] (t-[1])(t-[3])\cdots(t-[2k-1]),$$
    \item $c_B[n, k] = e_{n-k}([1], [3], \ldots, [2n-1])$ and
    $$\sum_{k=0}^n c_B[n, k] t^k = (t+[1])(t+[3])\cdots(t+[2n-1]),$$
\end{itemize}
where $h$ and $e$ are complete homogeneous and elementary symmetric polynomials.  Along the way, we prove a new identity, \Cref{t^n:x}, for the $h$ polynomials.

In \Cref{ci}, we provide combinatorial interpretations of the $S_B[n, k]$ and $c_B[n, k]$ (and hence also the case $q=1$) in terms of certain statistics on signed partitions and signed permutations, respectively; see Theorems~\ref{SB:inv}, \ref{SB:maj}, and~\ref{cB:inv}. In particular, we show that $c_B(n, k)$ counts the number of signed permutations with $k$ pairs of ``negative cycles.''

In \Cref{eegf}, we prove a variety of additional generating function identities, including the following:

\begin{align*}
    \sum_{n=0}^\infty S^o[n, k] \frac{x^n}{[n]!}
      &= \frac{1}{q^{\binom{k}{2}}}
  \sum_{i=0}^k (-1)^{k-i}q^{\binom{k-i}{2}}  \gau{k}{i} \exp_q([i] x), \\
    \sum_{n=0}^\infty S_B^o(n, k) \frac{x^n}{n!}
      &= e^x (e^{2x} - 1)^k, \\
    \sum_{n=0}^\infty S_B^o[n, k] \frac{x^n}{[n]!}
      &= \frac{1}{q^{k^2}}
  \sum_{i=0}^k (-1)^{k-i}q^{2\binom{k-i}{2}}  \gau{k}{i}_{q^2} \exp_q([2i+1] x).
\end{align*}
We also provide identities involving (bivariate) generating functions for $c[n, k]$ and $c_B[n, k]$. In each case, they satisfy a first or second order linear $q$-difference equation.

In \Cref{oai}, we prove the alternating sum identities
$$
\sum_{k=0}^n (-q)^{n-k} S^o[n, k] = 1 = \sum_{k=0}^n (-q)^{n-k} S_B^o[n, k],
$$
answering the conjecture in \cite{SW:hdf} in the affirmative. See \Cref{alt:sumA} and \Cref{alt:sumB}. We provide both algebraic demonstrations and combinatorial proofs based on sign-reversing involutions. Replacing $-q$ with $-q^m$ in the above sums conjecturally gives the graded Euler characteristic of certain generalized exterior derivative complexes on super coinvariant algebras. Our sign-reversing involutions more strongly give divisibility properties for these conjectured graded Euler characteristics. Furthermore, we conjecture bases for the type $A$ and $B$ super coinvariant algebras whose Hilbert series naturally give rise to $S^o[n, k]$ and $S_B^o[n, k]$ with our inversion-style statistics. See \Cref{ca} for details.  We end with a section of comments and many questions and open problems.

\section{Symmetric polynomials}
\label{sp}

Certain properties of the Stirling numbers follow from the fact that they can be expressed in terms of elementary and complete homogeneous symmetric polynomials.  We will derive some of them, including their ordinary generating functions, in this section.  Further information about symmetric polynomials can be found in the books of Macdonald~\cite{mac:sfh}, Sagan~\cite{sag:sg}, or Stanley~\cite{sta:ec2}.

\subsection{Ordinary generating functions}
\label{ogf}

Let $\bx=\{x_1,\ldots,x_n\}$ be a set of commuting variables.  The {\em degree $k$ elementary symmetric polynomial}, denoted $e_k(n)=e_k(x_1,\ldots,x_n)$, is the sum of all degree $k$ square-free monomials in $x_1,\ldots,x_n$.
The {\em degree $k$ complete homogeneous symmetric polynomial} $h_k(n)=h_k(x_1,\ldots,x_n)$ is the sum of all degree $k$ monomials in these variables.  For example
$$
e_2(3) = x_1 x_2 + x_1 x_3 + x_2 x_3
$$
and
$$
h_2(3) = x_1 x_2 + x_1 x_3 + x_2 x_3
+ x_1^2 + x_2^2 + x_3^2.
$$
By considering which monomials contain the variable $x_n$ and which do not, it is easy to derive the following recursions for $n\ge1$:
\beq
\label{e_n:rec}
e_k(n) = e_k(n-1) + x_n e_{k-1}(n-1)
\eeq
and
\beq
\label{h_n:rec}
h_k(n) = h_k(n-1) + x_n h_{k-1}(n).
\eeq

Let $t$ be a variable.  The following generating functions are well known and easy to prove directly from the definitions above:
\beq
\label{e_n:gf}
E_n(t):=\sum_{k=0}^n e_k(n) t^k = 
\prod_{i=1}^n (1 + x_i t)
\eeq
and
\beq
\label{h_n:gf}
H_n(t):=\sum_{k\ge0} h_k(n) t^k = 
\prod_{i=1}^n \frac{1}{1 - x_i t}.
\eeq

The next result is a type $B$ analogue of the following facts from the literature in type $A$:
\ben
\item[(a)] $c[n,k] = e_{n-k}([1],[2],\ldots,[n-1])$.
\item[(b)] $S[n,k] = h_{n-k}([1],[2],\ldots,[k])$.
\item[(c)] $\dil \sum_{k=0}^n c[n,k] t^k
=t(t+[1])(t+[2])\cdots(t+[n-1])$.
\item [(d)]  $\dil \sum_{n\ge k} S[n,k] t^n
=\frac{t^k}{(1-[1]t)(1-[2]t)\cdots(1-[k]t)}$.
\een

\bth
\label{cS:eh}
Let $n\in\bbN$ and $k\in\bbZ$.
\ben
\item[(a)] $c_B[n,k] = e_{n-k}([1],[3],\ldots,[2n-1])$.
\item[(b)] $S_B[n,k] = h_{n-k}([1],[3],\ldots,[2k+1])$.
\item[(c)] $\dil \sum_{k=0}^n c_B[n,k] t^k
=(t+[1])(t+[3])\cdots(t+[2n-1])$.
\item [(d)]  $\dil \sum_{n\ge k} S_B[n,k] t^n
=\frac{t^k}{(1-[1]t)(1-[3]t)\cdots(1-[2k+1]t)}$.
\een
\eth
\bprf
Parts (a) and (b) follow easily by comparing the recursions for the symmetric polynomials with those for the Stirlings and using induction.  Parts (c) and (d) are obtained by combining the first two parts with the generating functions~\eqref{e_n:gf} and~\eqref{h_n:gf}.
\eprf

\subsection{Falling factorials}

We will need a result about symmetric polynomials which, while not difficult to prove, we have not been able to find in the literature.  Given a variable $t$ and $k\in\bbN$ the corresponding {\em (type $A$) falling factorial} is
$$
t\da_k = t(t-1)(t-2)\cdots(t-k+1).
$$
It is well known that for the type $A$ Stirling numbers we have 
\beq
\label{t^n:A}
t^n =\sum_{k=0}^n S(n,k) t\da_k.
\eeq
To generalize this identity to symmetric polynomials in a set of variables $\bx$, let
$$
t\da_k^\bx = (t-x_1)(t-x_2)\cdots(t-x_k).
$$
\bth
\label{t^n:x}
For $n\in\bbN$,
$$
t^n = \sum_{k=0}^n h_{n-k}(k+1)\ t\da_k^\bx.
$$
\eth
\bprf
Induct on $n$, where the base case is easy to verify.  For the induction step we use equation~\eqref{h_n:rec} and pull off a factor from the falling factorial to get
\begin{align*}
\sum_k h_{n-k}(k+1)\ t\da_k^\bx
&=\sum_k h_{n-k}(k)\ (t-x_k)t\da_{k-1}^\bx +\sum_k x_{k+1} h_{n-k-1}(k+1)\ t\da_k^\bx\\
&=t\sum_k h_{n-k}(k)\ t\da_{k-1}^\bx -\sum_k x_k h_{n-k}(k)\ \da_{k-1}^\bx \\
&\rule{50pt}{0pt} +\sum_k x_k h_{n-k}(k)\ t\da_{k-1}^\bx\\
&=t\sum_k h_{(n-1)-k}(k+1) t\da_k^\bx \\
&=t\cdot t^{n-1}
\end{align*}
completing the proof.
\eprf

Clearly equation~\eqref{t^n:A} follows from the previous theorem by setting  $x_i = i-1$ for $1 \leq i \leq n+1$. More generally,
define the {\em (type $A$) $q$-falling factorial} to be
$$
(t;q)\da_k = t(t-[1])(t-[2])\cdots(t-[k-1]).
$$
Setting  $x_i = [i-1]$ in the previous theorem gives:
\begin{cor}
\label{t^n:Sq}
For $n\in\bbN$
$$
t^n = \sum_{k=0}^n S[n,k] (t;q)\da_k.
$$
\end{cor}

Now define the {\em type $B$ $q$-falling factorial} to be
$$
(t;q)\da_k^B = (t-[1])(t-[3])\cdots(t-[2k-1]).
$$
Setting $x_i = [2i-1]$ for $1 \leq i \leq k+1$ immediately gives an analogous type $B$ result.
\begin{cor}
\label{t^n:SBq}
For $n\in\bbN$
$$
t^n = \sum_{k=0}^n S_B[n,k] (t;q)\da_k^B,
$$
\end{cor}

\subsection{Inverse matrices}

For our matrix identity, we need a new result about symmetric polynomials.
Given a generating function $f(t)=\sum_{n\ge0} a_n t^n$ we define the coefficient extraction function to be
$$
[t^n] f(t) = a_n.
$$
A special case of the following  can be found in the book of
Sturmfels~\cite[p.\ 12]{stu:ait}.
\bth
We have
$$
\sum_{a+b=N} (-1)^a e_a(n) h_b(m) = \begin{cases}(-1)^N e_N(x_{m+1}, \ldots, x_n) & n \geq m, \\ h_N(x_{n+1}, \ldots, x_m) & n \leq m.\end{cases}
$$
\eth
\bprf
From the generating functions in~\eqref{e_n:gf} and~\eqref{h_n:gf} we obtain
\begin{align*}
\sum_{a+b=N} (-1)^a e_a(n) h_b(m)
&=[t^N] E_n(-t)H_m(t)\\
&=[t^N] \frac{(1-x_1 t)(1-x_2t)\cdots(1-x_nt)}{(1-x_1t)(1-x_2t)\cdots(1-x_mt)}\rule{0pt}{15pt}.
\end{align*}

If $n\ge m$ then we are finding the coefficient of $t^N$ in $(1-x_{m+1}t)(1- x_{m+2}t)\cdots(1-x_nt)$ which is the desired signed elementary symmetric polynomial.  The case $n\le m$ is similar.
\eprf

Define two infinite matrices with rows and columns indexed by $n\in\bbN$ and $k\in\bbN$, respectively, by
$$
E=[(-1)^{n-k}e_{n-k}(n)]_{n,k\ge0}
$$
and
$$
H=[h_{n-k}(k+1)]_{n,k\ge0}.
$$
Note that both $E$ and $H$ are lower uni-triangular.  Let $I$ be the $\bbN\times\bbN$ identity matrix.
\bth
We have
$$
EH = I
$$
\eth
\bprf
Since $E$ and $H$ are both lower uni-triangular, so is their product.  Therefore we only need to evaluate the entry $(EH)_{n,k}$ when $n>k$.  But
$$
(EH)_{n,k} = \sum_{i\ge 0} (-1)^{n-i}e_{n-i}(n) h_{i-k}(k+1)
$$
which is the sum in the previous theorem with $N=n-k$ and $m=k+1$.  So we are in the first case and the sum equals $(-1)^{n-k} e_{n-k}(x_{k+2},\ldots,x_n)=0$ since there are only $n-k-1$ variables in the elementary symmetric polynomial.
\eprf

Specialization  of $x_i=[2i-1]$ immediately gives the following result.

\bco 
If $s_B = [s_B[n,k]]_{n,k\ge0}$ and $S_B=[S_B[n,k]]_{n,k\ge0}$ then $s_BS_B = I$.\hqed 
\eco

Note that setting $x_i=i-1$  in the previous theorem gives the well-known result that $sS=I$ where
$s = [s(n,k)]_{n,k\ge0}$ and $S=[S(n,k)]_{n,k\ge0}$.

\section{Combinatorial interpretations}
\label{ci}

We will now give combinatorial interpretations of the $S_B[n,k]$ and $c_B[n,k]$ (and hence also the case $q=1$) in terms of certain statistics on signed partitions and permutations, respectively.  The lattice of signed partitions ordered by refinement is isomorphic to the intersection lattice $\cL_{B_n}$ for the hyperplane arrangement of the Coxeter group  $B_n$.  We will rederive the fact that $s_B(n,k)$ and $S_B(n,k)$ are the Whitney numbers of the first and second kind, respectively, for $\cL_{B_n}$, as well as proving a new result expressing the M\"obius function of $\cL_{B_n}$ in terms signed permutations.

\subsection{Signed partitions}

\begin{defn}
Let $S$ be a set and $\rho=\{S_1,\ldots,S_k\}$ be a set partition of $S$, so the $S_i$ are nonempty subsets whose disjoint union in $S$. We call the $S_i$ {\em blocks} and write $\rho=S_1/\ldots/S_k\ptn S$, removing the set braces from the $S_i$ themselves.
\end{defn}

\begin{defn}
\label{B_part}
A {\em signed} or {\em type $B$ partition} is a partition of the set 
$\spn{n}$
of the form
$$
\rho=S_0/S_1/S_2/\dots/S_{2k}
$$
satisfying
\ben
\item $0\in S_0$ and if $i\in S_0$ then $-i\in S_0$, and
\item for $i\ge1$ we have $S_{2i}=-S_{2i-1}$,
\een
where $-S=\{-s\ :\ s\in S\}$.
Call the blocks $S_{2i}$ and $S_{2i-1}$ {\em paired}.
We write $\rho\ptn_B \spn{n}$ and let 
$S_B(\spn{n},k)$ denote the set of all type $B$ partitions of $\spn{n}$ with $2k+1$ blocks.  
\end{defn}

For readability, we will sometimes use an overline to represent a negative sign and group the paired blocks together separated by a forward slash, while vertical slashes separate pairs.  Finally, set braces and commas may be removed.
\begin{exa}
To illustrate, an element of $S_B(\spn{7},2)$ is
$$
\rho=0 \ol{1} 1 \ol{3} 3 \ol{6} 6 \mid \ol{4}/4 \mid 2 \ol{5} 7/\ol{2} 5 \ol{7}.
$$
\end{exa}
Note that we may write the elements of any block in any order, reverse the order of any pair, and rearrange the pairs amongst themselves without changing the type $B$ partition.

Let $|S|=\{|s|\ :\  s\in S\}$, so that $|S_{2i}|=|S_{2i-1}|$ for $i\ge1$.
For all $i$ we let
$$
m_i=\min |S_i|.
$$
We will always write signed partitions in {\em standard form} which means that
\ben
\item $m_{2i}\in S_{2i}$ for all $i$, and
\item $0=m_0<m_2<m_4<\dots<m_{2k}$.
\een
\begin{exa}
The standard form of our example type $B$ partition above is
$$
\rho=0 \ol{1} 1 \ol{3} 3 \ol{6} 6 \mid \ol{2} 5 \ol{7}/2 \ol{5} 7 \mid \ol{4}/4
$$
with $m_0=0$, $m_1=m_2=2$, and $m_3=m_4=4$.
\end{exa}

\begin{defn}
\label{invB}
An {\em inversion} of $\rho \ptn_B \spn{n}$
written in standard form is a pair
$(s,S_j)$ satisfying
\ben
\item $s\in S_i$ for some $i<j$, and
\item $s\ge m_j$.
\een
Let $\Inv\rho$ be set of  inversions of $\rho$ and $\inv\rho=\#\Inv\rho$ where the hash tag denotes cardinality.
\end{defn}

While we have not specified the order of elements within each block, this does not affect which pairs form inversions.
Note that the second condition implies that any $s$ causing an inversion must be positive. Note also that in this condition it is not actually possible for $s=m_i$ because of the partition being in standard form.  However, we include it because when considering ordered set partitions equality will be possible and such cases will need to be counted.

\begin{exa}
Continuing our example, 
$$
\Inv\rho=\{(3,S_1), (3,S_2), (6,S_1), (6, S_2), (6,S_3), (6, S_4),
(5,S_2), (5,S_3), (5,S_4), (7,S_3), (7,S_4)\}
$$
so that $\inv \rho=11$.
\end{exa}

The $S_B[n,k]$ count signed partitions by inversions. The next theorem is a type $B$ analogue of the result of Milne~\cite{mil:rgf} cited in the introduction for $S[n,k]$. See also \Cref{sco_A} for a strongly related interpretation involving the Hilbert series of a conjectural monomial basis. 

\bth
\label{SB:inv}
We have
$$
S_B[n, k]=\sum_{\rho\in S_B(\spn{n},k)} q^{\inv \rho}.
$$
\eth
\bprf
We proceed by induction on $n$ where the base case is trivial.
Given $\rho\in S_B(\spn{n},k)$ we can remove $n$ and $-n$ to obtain a new partition $\rho'$.  

If $n$ (and thus $-n$) is in a singleton block then $\rho'\in S_B(\spn{n-1},k-1)$  and there is only one way to construct $\rho$ from $\rho'$.
Furthermore, in this case the standardization condition forces $S_{2k-1}=\{-n\}$ and $S_{2k}=\{n\}$ in $\rho$.  It follows that
$\inv\rho=\inv\rho'$. So, by induction, such $\rho$ contribute $S_B[n-1,k-1]$ to the sum.

If $n$ (and thus $-n$) is in a block with other elements, then $\rho'\in S_B(\spn{n-1},k)$.  The  possible $\rho$ giving rise to a fixed $\rho'$ are obtained by inserting $n$ in one of the $2k+1$ blocks of $\rho'$.  
And if $n$ is put in block $S_i$ then this adds inversions of the form $(n,S_j)$ for all $j>i$.  
Furthermore the placement of $-n$, wherever it is forced by that of $n$, does not contribute any inversions.
So
$\inv\rho=2k-i+\inv\rho'$ where $0\le i\le 2k$.
Thus the contribution of these $\rho$ is $[2k+1] S_B[n-1,k]$ and we are done.
\eprf

\begin{defn}
The {\em descent set} of $\rho\in S_B(\spn{n},k)$ is the multiset
$$
\Des\rho =\{\!\{1^{n_1}, 2^{n_2},\dots, (2k)^{n_{2k}}\}\!\}
$$
where $n_i$ is the number of elements of $S_{i-1}$ which are greater than $m_i$. Define the {\em major index} of $\rho$ to be
$$
\maj\rho = 1\cdot n_1 + 2\cdot n_2+\cdots+2k \cdot n_{2k}.
$$
\end{defn}

Note that this convention differs from the one for descents in a permutation in the symmetric group since $i$ is the index of the block containing the smaller integer.  This could be fixed by renumbering the blocks, but then the conventions above for $S_0$ would become less natural.

\begin{exa}
In our perennial example, descents are caused by the $3$ and $6$ in $S_0$, the $5$ in $S_1$, and the $7$ in $S_2$.  Hence  $\Des\rho=\{\!\{1^2,2^1,3^1\}\!\}$ and
$\maj\rho=1\cdot 2 + 2\cdot 1 +3\cdot 1= 7$.
\end{exa}
Just as with permutations in the symmetric group, $\inv$ and $\maj$ have the same distribution over signed partitions.
\bth
\label{SB:maj}
We have
$$
S_B[n,k]=\sum_{\rho\in S_B(\spn{n},k)} q^{\maj \rho}.
$$
\eth
\bprf
We proceed as in the previous proof, keeping the same notation.  When $n$ and $-n$ are in singleton blocks, the same reasoning applies to show $\maj\rho=\maj\rho'$. So, as before, such $\rho$ contribute $S_B[n-1,k-1]$ to the sum.
Now assume that $n$ is inserted into a block of $\rho'$ and $-n$ into the companion block.
If $n$ is in $S_i$ then 
$$
\maj\rho=
\case{\maj\rho'}{if $i=2k$,}{i+1+\maj\rho'}{if $0\le i<2k$.}
$$
Thus the contribution of these $\rho$ is $[2k+1] S_B[n-1,k]$ which finishes the proof.
\eprf

It is possible to give a combinatorial proof of equation~\eqref{t^n:A} by showing that when $t$ is a positive integer both sides count the set of functions $f:[n]\ra[t]$.  We will now give a combinatorial proof of \Cref{t^n:SBq} when $q=1$ using similar ideas.

\begin{defn}
A {\em type $B$ function} is any function $f:\spn{n}\ra\spn{p}$ satisfying
\beq
\label{Bfcn}
f(-i)=-f(i)
\eeq
for all $i\in\spn{n}$. The {\em kernel of $f$} is the partition $\ker f$ of $\spn{n}$ whose blocks are the nonempty fibers $f^{-1}(j)$ for $j\in\spn{p}$.
\end{defn}

In particular, a type $B$ function satisfies $f(0)=0$, and if $f(i)=0$, then $f(-i)=0$. Hence the definition of a type $B$ function ensures that $\ker f$ is a type $B$ partition.

Define the \textit{type $B$ falling factorial} to be
$$
t\da_k^B = (t;1)\da_k^B=(t-1)(t-3)\cdots(t-2k+1).
$$
The next result is the special case $q=1$ of \Cref{t^n:SBq}.  But here we give a combinatorial proof rather than relying on symmetric functions.
We note that this result with essentially the same proof below was also given in a paper of Bagno and Garber~\cite{BG:spb}, but we include it for completeness.
\bth
For $n\in\bbN$,
$$
t^n = \sum_k S_B(n,k) t\da^B_k.
$$
\eth
\bprf
Since this is a polynomial identity, it suffices show that it holds when $t$ is an odd positive integer, say $t=2p+1$.  We claim that in this case both sides of the identity count the number of type $B$ functions $f:\spn{n}\ra\spn{p}$.

On the one hand, we can determine $f$ by first choosing $f(1),\dots,f(n)$, in which case $f(-1),\dots,f(-n)$ are determined by~\ree{Bfcn}. 
And we know $f(0)=0$.
Since $\#\spn{p}=t$, there are $t$ ways to choose each of the necessary values, for a total count of $t^n$.

Alternatively, we can construct $f$ by first picking a type $B$ partition $\rho=S_0/S_1/\dots/S_{2k}$ of $\spn{n}$ to be $
\ker f$ and then injectively mapping the blocks of $\rho$ into $\spn{p}$.  Since $0\in S_0$ we must have $f(S_0)=0$.
This leaves $t-1$ choices for $f(S_1)$.  Now $f(S_2)$ is determined by~\ree{Bfcn}.  There remains $t-3$ choices for $f(S_3)$, and so forth.
\eprf

\subsection{Signed permutations}

We now turn to the Stirling numbers of the first kind.  A {\em permutation} of a finite set $S$ is a bijection $\pi:S\ra S$.  As usual, $\pi$ can be factored into {\em cycles} $c=(a_1,a_2,\ldots,a_k)$ where $\pi(a_i)=a_{i+1}$ with subscripts taken modulo $k$.  Let $\spn{n}'=\spn{n}\setm\{0\}$.

\begin{defn}
A {\em signed} or {\em type $B$ permutation} is a permutation $\pi$ of $\spn{n}'$ satisfying
\begin{equation}
  \label{B_nPer}
  \pi(-i)=-\pi(i)
\end{equation}
for all $i\in\spn{n}'$.

This condition implies that any cycle $c$ of $\pi$ is of one of two types.
\ben
\item  If $c=(a_1,a_2,\ldots,a_k)$ does not contain both $i$ and $-i$ for any $i\in\spn{n}'$ then $\pi$ also contains $-c = (-a_1,-a_2,\ldots,-a_k)$.  We say that $c$ and $-c$ are {\em paired}.
\item If $c$ contains both $i$ and $-i$ for some $i\in\spn{n}'$ then $c$ must have the form
$$
c=(a_1,a_2,\ldots,a_k,-a_1,-a_2,\ldots,-a_k).
$$
We call such a cycle {\em unpaired}.
\een
Finally, let $c_B(\spn{n}',k)$ be the set of all $B_n$ permutations with $2k$ paired cycles.
\end{defn}

We use the same conventions when writing signed cycle decomposition as for partitions, using bars instead of negative signs and keeping paired cycles closer together than others.

\begin{exa}
In the permutation
$$
\pi = (1,\ol{3},\ol{1},3)\ (\ol{4})(4)\ (2,\ol{5},7)(\ol{2},5,\ol{7})\ (\ol{6},6)
$$
the four cycles $(\ol{4})$, $(4)$, $(2,\ol{5},7)$, and $(\ol{2},5,\ol{7})$ are paired while the cycles $(1,\ol{3},\ol{1},3)$ and $(\ol{6},6)$ are not. Since there are $k=2$  cycle pairs, $\pi\in c_B(\spn{7}',2)$.
\end{exa}

We will use the absolute value and minimum notation for permutations exactly as we did for partitions.  Our {\em standard form} for a $B_n$ permutation $\pi=c_1 c_2\ldots c_\ell$ will be to list the cycles so that the minima $m_i=\min|c_i|$ satisfy
\ben
\item $m_1\le m_2\le\ldots\le m_\ell$,
\item  if $m_i=m_{i+1}$ then $-m_i\in c_i$ and 
$m_{i+1}\in c_{i+1}$,
\item each $c_i$ is listed with its $\pm m_i$ last, with unpaired cycles ending in $-m_i$. 
\een
Putting our example permutation is standard form gives
$$
\pi = (3,1,\ol{3},\ol{1})\ (5,\ol{7},\ol{2})(\ol{5},7,2)\ (\ol{4})(4)\ (6,\ol{6}).
$$

\begin{defn}
Given $\pi\in c_B(\spn{n}',k)$ in standard form, we let $w=w_1 w_2\ldots w_{2n}$ be the word obtained by removing the parentheses from $\pi$. Define the {\em set of inversions of $\pi$} to be 
$$
\Inv \pi = \{ (i,j)\ :\ \text{$i<j$ and $w_i>|w_j|$}\}
$$
with corresponding {\em inversion number}
$$
\inv \pi = \#\Inv\pi.
$$
\end{defn}

\begin{exa}
For our running example permutation, we have
$$
w = 3,1,\ol{3},\ol{1}, 5,\ol{7},\ol{2}, \ol{5},7,2, \ol{4}, 4, 6,\ol{6}.
$$
Hence $\Inv\pi$ is the set of pairs
$$
(1, 2), (1, 4), (1, 7), (1, 10), (5, 7), (5, 10), (5, 11), (5, 12), (9, 10), (9, 11), (9, 12),(9, 13), (9, 14)
$$
and $\inv\pi = 13$.
\end{exa}

\bth
\label{cB:inv}
We have
$$
c_B[n,k]=\sum_{\pi\in c_B(\spn{n}',k)} q^{\inv\pi}.
$$
\eth
\bprf
As usual, we induct on $n$ and only give details for the induction step.  Take $\pi\in c_B(\spn{n}',k)$ and remove $n$ and $-n$ to form $\pi'$.

If $n$ and $-n$ were both fixed points, then 
$\pi'\in c_B(\spn{n-1}',k-1)$.  Furthermore, standard form forces the last two cycles of $\pi$ to be $(-n)(n)$ so that 
$\inv\pi=\inv \pi'$.  Thus permutations in this case give a contribution of $c_B[n-1,k-1]$ to the sum.

The other possibility is that $n$ and $-n$ are both in cycles of length at least two.  If these cycles are paired, then they remain paired cycles after these elements are removed.  If both elements are in the same unpaired cycle then, even if that cycle contains no other elements, upon removal $\pi'$ still has the same number of paired cycles.  So in either case $\pi'\in c_B(\spn{n-1}',k)$.

Now consider all the ways $\pm n$ can be inserted in a given $\pi'$  in this case.  Note that the position of $n$ determines the position of $-n$, and that $-n$ can never cause any inversions.
One possibility is to adjoin the cycle $(n,-n)$ to $\pi$ which must be at the right end to be in standard form.  Now $n$ causes no inversions either so $\inv\pi=\inv\pi'$.  The other possibility is that $n$ is inserted just before, and in the same cycle as, any of the $2n-2$ elements of $\pi'$.  (One must use the space before since using the space after a final element in a cycle would make the result nonstandard.)  If this element is the $i$th from the right then $\inv\pi=i+\inv\pi'$ where $1\le i\le 2n-2$.  It follows that the total contribution of this case is
$$
(q^0 + q^1 +\cdots + q^{2n-2}) c_B[n-1,k] = [2n-1]  c_B[n-1,k]
$$
which finishes the proof.
\eprf

Letting $q=1$ in the previous theorem gives the following result.

\begin{cor}\label{cB:inv_q=1}
  The number of $B_n$ permutations with $2k$ paired cycles is $c_B(n, k)$.
\end{cor}

\subsection{The lattice of signed partitions}
\label{lsp}

We now connect the $S_B(n,k)$ and $s_B(n,k)$ with the intersection lattice for the Coxeter group $B_n$.

\begin{defn}
Let $P$ be a finite poset (partially ordered set) with a unique minimal element $\zh$.  Call $P$ {\em ranked} if, for every $x\in P$, all maximal chains from $\zh$ to $x$ have the same length.  This length is called the {\em rank of $x$} and denoted $\rk x$.  The {\em $k$th rank of $P$} is
$$
\Rk(P,k) = \{ x\in P\ :\ \rk x = k\}
$$
with corresponding {\em Whitney number of the second kind}
$$
W(P,k) = \# \Rk(P,k).
$$
The {\em (one variable)  M\"obius function} of $P$ is the function $\mu:P\ra\bbZ$ defined recursively by
$$
\sum_{x\le y} \mu(x) = \de_{\zh,y}.
$$
This is a far-reaching generalization of the M\"obius function in number theory.  See~\cite{sag:aoc} or~\cite{sta:ec1} for more details.  The {\em Whitney numbers of the first kind} for $P$ are
$$
w(P,k) = \sum_{x\in\Rk(P,k)} \mu(x).
$$
If $P$ has a unique maximal element $\oh$ then we will use the notation
$$
\mu(P) = \mu(\oh).
$$
\end{defn}

\begin{defn}
For any finite set $S$ we denote by $\Pi_S$ the lattice of all  set partitions of $S$ ordered by {\em refinement} so that $\rho\le \si$ if every block of $\rho$ is contained in some block of $\si$.
Let $\Pi_{B_n}$ denote the subposet of $\Pi_{\spn{n}}$ obtained by restricting the partial order to the type $B$ partitions.  We  call $\Pi_{B_n}$ the {\em $B_n$ partition lattice}. 
\end{defn}

The Coxeter group $B_n$ has reflecting hyperplanes $x_i=0$ and $x_i=\pm x_j$ where $i,j\in[n]$ and $x_i$ is the $i$th coordinate function.  The corresponding {\em intersection lattice}, $\cL_{B_n}$, is the set of all subspaces which are intersections of these hyperplanes ordered by reverse inclusion.   Zaslavsky~\cite{zas:grs} showed that these subspaces are in bijective correspondence with certain signed graphs which are clearly in bijection with signed set partitions of $\spn{n}$.  It easily follows that $\cL_{B_n}$ and $\Pi_{B_n}$ are isomorphic.

It will be useful to connect type $B$ partitions and permutations.  Given a signed permutation $\pi$, its {\em underlying signed set partition} is $\rho$ obtained by replacing every paired cycle by its underlying subset, and taking the union of the sets underlying all paired cycles together with $\{0\}$ to be the zero block.  The reader can verify that our recurring example $\pi$ and $\rho$ satisfy this relation.  We let $B(\rho)$ be the set of type $B$ permutations with underlying set $\rho$.
The first and third parts of the next result also follow from~\cite{zas:grs}.
\bth
\label{Pi_B_n}
Fix $n\in\bbN$ and let $\mu$ be the M\"obius function of $\Pi_{B_n}$.
\ben
\item[(a)] For $0\le k\le n$, we have $W(\Pi_{B_n},k) = S_B(n,n-k)$.
\item[(b)] If $\rho=S_0/S_1/S_2/\ldots/S_{2k}\in S_B(\spn{n},k)$ then 
\begin{align*}
 \#B(\rho)&=(\#S_0-2)!! \prod_{i=1}^k (\#S_{2i}-1)!\\
 &=(-1)^{n-k}\mu(\rho).   
\end{align*}
\item[(c)]  For $0\le k\le n$ we have $w(\Pi_{B_n},k) = s_B(n,n-k)$.
\een
\eth
\bprf
To prove (a), if $\rho$ is covered by $\si$ in $\Pi_{B_n}$, then there are two possibilities. One is that there were two pairs of blocks $S_{2i-1}, S_{2i}$ and $S_{2j-1},S_{2j}$ in $\rho$ which were replaced in $\si$ by a pair whose component blocks are unions of one block from each of the given pairs.  The other is that a pair $S_{2i-1}, S_{2i}$ was absorbed into the zero block.
Note that in either case, the total number of blocks decreases by two in passing from $\rho$ to $\si$.
Using this fact, the desired result now follows from an easy induction on $k$.

For the first equality in (b), note that, for $i\ge1$, the number of ways to turn $S_{2i}$ into a cycle is $(\#S_{2i}-1)!$.  
And once a cycle is put on $S_{2i}$, the cycle on $S_{2i-1}$ is fixed.
So, letting $\#S_0=2m+1$, it suffices to show that the number of ways to decompose $S_0\setm\{0\}$ into unpaired cycles is
$(2m-1)!!$.  But the count we seek is the number of signed permutations of $\spn{m}'$ which have no paired cycles. By \Cref{cB:inv_q=1} and \Cref{cS:eh}(a) with $q=1$, this is $c_B(m, 0) = e_m(1, 3, \ldots, 2m-1) = (2m-1)!!$.

We now prove  (c) and the second equality in (b) simultaneously by induction on $n$.  From the description of the covering relations in part (a) we see that
the interval $[\zh,\rho]$ in $\Pi_{B_n}$ is isomorphic to a product of posets, with one poset for $S_0$ and one for each pair $S_{2i-1},S_{2i}$ for $1\le i\le k$.  For $i\ge1$, the partitions contained in $S_{2i}$ form the lattice $\Pi_{S_{2i}}$.  And once a partition $\si$ of $S_{2i}$ is chosen, then $S_{2i-1}$ must be partitioned so that each block is the negative of some block of $\si$.  For the zero block, the partitions contained in $S_0$ contribute the lattice $\Pi_{B_m}$. Thus  we have the isomorphism
$$
[\zh,\rho] \iso \Pi_{B_m} \times  \Pi_{S_2}  \times \Pi_{S_4}\times  \cdots \times \Pi_{S_{2k}}
$$
and, since $\mu(P\times Q)=\mu(P)\mu(Q)$,
\beq
\label{mu(rho)}
\mu(\rho) = \mu(\Pi_{B_m}) \prod_{i=1}^{k} \mu(\Pi_{S_{2i}}).
\eeq
It is well known that $\mu(\Pi_S) =(-1)^{\#S-1}(\#S-1)!$.   And we can assume by induction that for $m<n$ we have $\mu(\Pi_{B_m})=(-1)^m (2m-1)!!$.  Plugging these values into~\eqref{mu(rho)} proves the second equality in (b) as long as $\rho<\oh$.  Furthermore, from the proof of  (a), the  $\rho$ at rank $k$ of $\Pi_{B_n}$ are exactly those signed partitions of $\spn{n}$ with $2(n-k)+1$ blocks. It follows that for $\rho<\oh$, the sum of the $\mu$ values of these partitions is exactly $(-1)^k c_B(n,n-k)=s_B(n,n-k)$.

To handle the case $\rho=\oh$, we use the definition of $\mu$ and the previous paragraph to give
$$
\mu(\oh) = -\sum_{\rho<\oh} \mu(\rho) = -\sum_{k<n} s_B(n,n-k).
$$
On the other hand, plugging in $q=1$ and $t=-1$ to \Cref{cS:eh}(c) and multiplying both sides by $(-1)^n$ gives
$$
\sum_{k=0}^n s_B(n,k) = 0.
$$
Comparing the last two displayed equations and using \Cref{cS:eh}(a) when
$q=1$ and $k=0$ yields
$$
\mu(\oh) = s_B(n,0)= (-1)^n e_n(1,3,\ldots,2n-1) = (-1)^n (2n-1)!!
$$
which finishes this case and the proof.
\eprf

\section{Exponential and $q$-exponential generating functions}
\label{eegf}

We now derive exponential and $q$-exponential generating functions for the Stirling numbers and $q$-Stirling numbers of types $A$ and $B$.  We give combinatorial proofs for the former using the theory of species.  See the book of Bergeron, Labelle, and Leroux~\cite{BLL:cst} or~\cite[Chapter 4]{sag:aoc} for more information.  For the latter, we use the theory of $q$-difference equations.
Existence and uniqueness of solutions of linear $q$-difference equations with constant coefficients is given explicitly in \cite[Thm.~2.11]{MR2963764}. A general existence and uniqueness result for $q$-difference equations is given in \cite[Thm.~2.1]{MR2963764}. Basic properties of $q$-derivatives such as their definition and the $q$-product rule are summarized in \cite[\S1.3]{MR2963764}. 

Our $q$-exponential generating function identities do not seem to be known for the Stirling numbers even for $S[n,k]$ and $s[n,k]$. Consequently, we will provide full proofs of these results in type $A$ and then just sketch whatever changes are needed for type $B$.

\subsection{Exponential generating functions}

The following identities are well known; see e.g.~\cite[\S26.8(ii)]{NIST:DLMF}, \cite[A008275]{oeis}, and \cite[A008277]{oeis}. 
\begin{align*}
\dil\sum_{n\ge0} S(n,k) \frac{x^n}{n!} &= \frac{1}{k!} (e^x-1)^k,\\
\dil \sum_{n,k\ge0} S(n,k) t^k \frac{x^n}{n!} &= e^{t(e^x-1)},\\
\dil\sum_{n\ge0} c(n,k) \frac{x^n}{n!} &= \frac{1}{k!} \left(\ln\frac{1}{1-x}\right)^k,\\
\dil \sum_{n,k\ge0} c(n,k) t^k \frac{x^n}{n!} &= \frac{1}{(1-x)^t}.
\end{align*}
The type $B$ analogues are as follows. The first is stated (without proof) in \cite[A039755]{oeis}. The last two are stated (again, without proof) in \cite[A028338]{oeis}, though effectively using \Cref{cS:eh}(c) as a definition. 
\bth
\label{egf:thm}
We have
\ben
\item[(a)] $\dil\sum_{n\ge 0} S_B(n,k) \frac{x^n}{n!}
=\frac{1}{2^k k!} e^x(e^{2x}-1)^k$.\\
\item[(b)] $\dil \sum_{k,n\ge0} S_B(n,k) t^k \frac{x^n}{n!}
= e^x \sqrt{e^{t(e^{2x}-1)}}$.\\
\item[(c)] $\dil\sum_{n\ge 0} c_B(n,k) \frac{x^n}{n!}
=\frac{1}{k!\sqrt{1-2x}} \left(\log\frac{1}{\sqrt{1-2x}}\right)^k$.\\
\item[(d)] $\dil\sum_{k,n\ge0} c_B(n,k) t^k \frac{x^n}{n!}
= \left(\frac{1}{\sqrt{1-2x}}\right)^{1+t}$.
\een
\eth
\begin{proof}
We will only prove (a) and (c) since then (b) and (d) follow by summing on $k$.

For (a) note there is a bijection between type $B$ partitions $S_0/\ldots/S_{2k}$
and ordered pairs of the form 
$(T_0,S_2/S_4/\ldots/S_{2k})$ where $T_0$ is the (possibly empty) set of positive integers in $S_0$.  This is because the other elements of $S_0$ are $-S_0\uplus\{0\}$ and $S_{2i-1}=-S_{2i}$ for all $i$.

By the conventions for the $S_{2i}$, its smallest element is positive and the others can be signed arbitrarily.  Consider the species $\cS$ such that $\cS(L)$ is all sets obtained from a set of nonempty integers $L$ by arbitrarily signing every element of $L$ except the smallest.  The number of such sets is $2^{\#L-1}$ with corresponding exponential generating function
$$
\sum_{n\ge1} 2^{n-1} \frac{x^n}{n!}=\frac{1}{2}(e^{2x}-1).
$$
So, by the Product Rule for exponential generating functions~\cite[Theorem 4.4.2(b)]{sag:aoc}, the exponential generating function for ordered $k$-tuples of such sets is $(e^{2x}-1)^k/2^k$.  We must divide by $k!$ to remove the order. For $T_0$, we use the species $\cT(L) = \{L\}$  which has exponential generating function $e^x$ since $\#\cT(L)=1$ for all $L$.  Now using the Product Rule again completes the proof of (a).

To prove (c), consider a type $B$ permutation $\pi$ and its underlying partition $\rho$.  As in part (a), we can reconstruct $\rho$ from a pair $(T_0,S_2/S_4/\ldots/S_{2k})$.  To recover the possible  $\pi$ associated with $\rho$ we must sign all but the smallest element of $S_{2i}$ and then put a cycle on these elements, where the latter can be done in $(n-1)!$ ways if $\#S_{2i}=n$.  This gives the exponential generating function
$$
\sum_{n\ge1} (n-1)! 2^{n-1} \frac{x^n}{n!}
=\frac{1}{2}\sum_{n\ge1} \frac{(2x)^n}{n} =
\log\frac{1}{\sqrt{1-2x}}.
$$

On the other hand, if $\#T_0=n$ then the  elements in $T_0\uplus (-T_0)$  need to be turned into unpaired cycles.  By \Cref{cS:eh}(a), this can be done in
$c_B(n,0) = (2n-1)!!$ ways which gives an exponential generating function of
$$
\sum_{n\ge0} (2n-1)!! \frac{x^n}{n!}
=\sum_{n\ge0}  (-1/2)\da_n (-2)^n\frac{x^n}{n!}
=\sum_{n\ge0} \binom{-1/2}{n}(-2x)^n = \frac{1}{\sqrt{1-2x}}.
$$
Using the Product Rule and dividing by $k!$ to remove the order finishes the proof.
\eprf

\subsection{$q$-exponential generating functions}
\label{qegf}

For our $q$-exponential generating functions, we will need the {\em $q$-binomial coefficients}
$$
\gau{n}{k}=\frac{[n]!}{[k]![n-k]!},
$$
the {\em $q$-exponential function}
$$
\exp_q(x) = \sum_{n\ge0} \frac{x^n}{[n]!},
$$
the {\em $q$-logarithm}
$$
-\log_q(1-x) = \sum_{n=1}^\infty \frac{x^n}{[n]},
$$
as well as the {\em $q$-derivative}
$$
D_q f(x) = \frac{f(qx)-f(x)}{qx-x}.
$$
See \cite[\S1.3]{MR2963764} for a summary of $q$-calculus. It will also be convenient to find the $q$-exponential generating functions for the ordered versions of the Stirling numbers of the second kind.
\bth
We have
\beq
\label{S^o:Egf}
  \sum_{n\ge0} S^o[n, k] \frac{x^n}{[n]!} = 
  \frac{1}{q^{\binom{k}{2}}}
  \sum_{i=0}^k (-1)^{k-i}q^{\binom{k-i}{2}}  \gau{k}{i} \exp_q([i] x),
\eeq
and
$$
     \sum_{n, k \geq 0}  S^o[n, k] t^k\frac{x^n}{[n]!} 
      = \sum_{i=0}^\infty \frac{q^i t^i}{(1+t)(q+t)\cdots (q^i+t)} \exp_q([i] x).
$$
\eth
\bprf
To simplify the proof of the first equality, let
$$
E_k = \sum_{n\ge0} S^o[n, k] \frac{x^n}{[n]!}
=\sum_{n\ge0} [k]! S[n, k] \frac{x^n}{[n]!}.
$$
From the recursion~\eqref{eq:Sq_rec} we get that
\beq\label{S^o:rec}
S^o[n,k] = [k](S^o[n-1,k-1] + S^o[n-1,k])
\eeq
for $n\ge1$.
Combining this with  the $q$-derivative
$$
D_q x^n = [n] x^{n-1}
$$
implies that
\beq
\label{D_qE_k}
D_q E_k = [k] (E_k + E_{k-1})
\eeq
for all $k\in\bbZ$  with $E_k = 0$ for $k < 0$.  This equation and induction on $k$ give
$$
\left(\prod_{i=0}^k (D_q - [i])\right) E_k = 0.
$$
Since $D_q \exp_q([i] x) = [i] \exp_q([i] x)$, the theory of linear, constant-coefficient $q$-difference equations now implies that
\beq
\label{E_k}
E_k = \sum_{i=0}^k c_i^k \exp_q([i] x)
\eeq
for certain $c_i^k$ which are constant with respect to $x$. So it suffices to show that
\beq
\label{c_i^k.0}
c_i^k = \frac{(-1)^{k-i}q^{\binom{k-i}{2}}}{q^{\binom{k}{2}}}  \gau{k}{i}.
\eeq

Substituting~\eqref{E_k} into~\eqref{D_qE_k} and using the linear independence of $\{\exp_q([i] x)\}_{i \in \bbN}$ gives
$$
c_i^k = \frac{-[k]}{q^i [k-i]} c_i^{k-1}
$$
for $0 \leq i \leq k-1$ which, after iteration, results in
\beq
\label{c_i^k}
c_i^k = \frac{(-1)^{k-i}}{q^{i(k-i)}} \gau{k}{k-i} c_i^i.
\eeq
Comparing this with~\eqref{c_i^k.0}  and using the identity 
\beq
\label{binoms}
i(k-i) + \binom{k-i}{2} = \binom{k}{2} - \binom{i}{2},
\eeq
we see that we will be done if we can show 
$c_i^i = q^{-\binom{i}{2}}$
for $i\ge0$. Let us rewrite this as $c_k^k = q^{-\binom{k}{2}}$ and induct on $k$.  Since $E_k(0)=\delta_{0,k}$, we see from~\eqref{E_k} that $c_0^0 = 1 = q^{-\binom{0}{2}}$ and, 
for $k \geq 1$,
$$
\sum_{i=0}^k c_i^k =0.
$$
This equation determines $c_k^k$ in terms of the $c_i^k$ for $i<k$.  And the latter are known to have the desired form by equation~\eqref{c_i^k} and induction.  So it suffices to prove that
$$
0=\sum_{i=0}^k  \frac{(-1)^{k-i}}{q^{i(k-i)+\binom{i}{2}}} \gau{k}{k-i}
=\frac{1}{q^{\binom{k}{2}}}\sum_{i=0}^k (-1)^{k-i} q^{\binom{k-i}{2}} \gau{k}{k-i}
=\frac{1}{q^{\binom{k}{2}}}\sum_{i=0}^k (-1)^i q^{\binom{i}{2}} \gau{k}{i}.
$$
But this last sum is seen to be zero by substituting $t=-1$ into the $q$-Binomial Theorem
\beq
\label{qBT}
\prod_{i=0}^{k-1} (1+q^i t) = \sum_{i=0}^k q^{\binom{i}{2}} \gau{k}{i} t^i.
\eeq

For the second equality in the statement of the theorem, we multiply~\eqref{S^o:Egf} by $t^k$ and sum to get
  \begin{align*}
    \sum_{n, k \geq 0} [k]! S[n, k] t^k \frac{x^n}{[n]!} 
      &= \sum_{k=0}^\infty \frac{t^k}{q^{\binom{k}{2}}}
     \sum_{i=0}^k (-1)^{k-i}q^{\binom{k-i}{2}}  \gau{k}{i} \exp_q([i] x) \\
      &= \sum_{i=0}^\infty \frac{1}{q^{\binom{i}{2}}} \left(\sum_{k=i}^\infty  \frac{(-1)^{k-i}}{q^{i(k-i)}} \gau{k}{i} t^k\right) \exp_q([i] x) \\
      &= \sum_{i=0}^\infty \frac{q^i t^i}{(1+t)(q+t) \cdots (q^i+t)} \exp_q([i] x),
  \end{align*}
  where the second equality uses equation~\eqref{binoms}, and the third is a form of the Negative $q$-Binomial Theorem as in Exercise 8(b) from Chapter 3 of~\cite{sag:aoc} substituting $-t$ for $t$ and $1/q$ for $q$.
\eprf

The proof of the next result is similar to the one just given.  So we will only mention the highlights.
\bth
We have
$$
  \sum_{n\ge0} S_B^o[n, k] \frac{x^n}{[n]!} = 
  \frac{1}{q^{k^2}}
  \sum_{i=0}^k (-1)^{k-i}q^{2\binom{k-i}{2}}  \gau{k}{i}_{q^2} \exp_q([2i+1] x),
$$
and
$$
     \sum_{n, k \geq 0}  S_B^o[n, k] t^k\frac{x^n}{[n]!} 
      = \sum_{i=0}^\infty \frac{q^{2i+1} t^i}{(q+t)(q^3+t)\cdots (q^{2i+1}+t)} \exp_q([2i+1] x).
$$
\eth
\bprf
Let $F_k$ be the first sum.  Using the recursion~\eqref{eq:SBq_rec} and the fact that 
$S_B^o[n,k]=[2k]!!S_B[n,k]$ we see that
$$
S_B^o[n,k] = [2k] S_B^o[n-1,k-1] + [2k+1] S_B^o[n-1,k]
$$
for $n\ge1$,  This in turn implies that for $k\in\bbZ$ 
\beq
\label{D_qF_k}
D_q F_k = [2k+1] F_k + [2k] F_{k-1}
\eeq
and $F_k=0$ for $k<0$.  Thus
$F_k$ satisfies the linear $q$-difference equation
$$
\left(\prod_{i=0}^k (D_q - [2i+1])\right) F_k = 0.
$$
So we can write
$$
F_k = \sum_{i=0}^k d_i^k \exp_q([2i+1] x)
$$
for certain constants $d_i^k$.  To show that the $d_i^k$ have the correct form, one uses~\eqref{D_qF_k} and iteration to get
$$
d_i^k = \frac{(-1)^{k-i}}{q^{(2i+1)(k-i)}} \gau{k}{k-i}_{q^2} d_i^i.
$$
The analogue of \eqref{binoms} is
\beq
\label{binoms_B}
(2i+1)(k-i) + i^2 = k^2 - 2\binom{k-i}{2},
\eeq
which implies that we must show $d_k^k = q^{-k^2}$. This is accomplished recursively using the $q$-Binomial Theorem~\eqref{qBT} as before with the substitutions $t=-1$ and $q^2$ for $q$. 

The bivariate generating function is now obtained by summing over $k$ and using the previous version of the Negative $q$-Binomial Theorem with substitutions $-t/q$ for $t$ and $1/q^2$ for $q$.
\eprf

For the $q$-Stirling numbers of the first kind we will need to use a version of the chain rule.  Unfortunately, no such analogue exists for $D_q(g(f))$.  But Gessel~\cite{ges:qae} defined a $q$-analogue of composition which does obey a $q$-chain rule.  Given an $q$-exponential generating function $f(x)$ with $f(0)=1$, define its {\em $k$th symbolic power} recursively by $f^{[0]}=1$ and 
$$
D_q f^{[k]} = [k] f^{[k-1]} D_q f.
$$
Note that $x^{[k]}=x^k$, so that when $q=1$ we have $f^{[k]}=f^k$
for $k\ge0$.  Given $g=\sum_{n\ge0} g_n x^n/[n]!$, Gessel then defines a {\em $q$-analogue of functional composition} to be
$$
g[f] = \sum_{n\ge0} g_n \frac{f^{[n]}}{[n]!}.
$$
Again, when $q=1$ we have $g[f]=g(f)$.  Gessel's $q$-analogue of the chain rule states that
\beq
\label{qcr}
D_q(g[f]) = (D_q g)[f]D_qf.
\eeq

The first result of the next theorem was also obtained by Johnson~\cite[(4.12)]{joh:saq}.
\bth
\label{cnk_egf}
We have
$$
    \sum_{n\ge0} c[n, k] \frac{x^n}{[n]!} = \frac{(-\log_q(1-x))^{[k]}}{[k]!},
$$
and
$$
    \sum_{n, k \geq 0} c[n, k] t^k\frac{x^n}{[n]!}  = \exp_q[-t\log_q(1-x)].
$$
\eth
\bprf
Let $C_k = \sum_{n\ge0} c[n, k] x^n/[n]!$.
The usual manipulations and the recursion~\eqref{eq:cq_rec} give the $q$-difference equation
$$
D_q C_k = \frac{C_{k-1}}{1-x}
$$
for $k\in \bbZ$ with  $C_k = 0$ for $k < 0$.  The formula for $C_k$ now follows from a simple induction on $k$ using the definition of symbolic power.  And the bivariate generating function is a consequence of the definition of $q$-composition.
\eprf

Unfortunately, for the  $c_B[n,k]$ we were only able to derive a differential equation for the desired $q$-exponential generating function.  We will have more to say about this in \Cref{qec}.

\section{Ordered analogues and identities}
\label{oai}

In this section we will prove alternating sum identities as well as divisiblity results for the ordered $q$-Stirling numbers of the second kind.  The former will prove two conjectures of Swanson and Wallach \cite{SW:hdf}.  Our main tools will be the use of sign-reversing involutions.  These results and their demonstrations are new even in type $A$ so, as in the previous section, the type $B$ proofs will only be sketched.

\subsection{Alternating sums}
\label{as}

We first need a combinatorial interpretation for the $S^o[n,k]$.  These polynomials count {\em ordered set partitions of $[n]$ into $k$ blocks} which are sequences of nonempty sets $\om= (S_1/S_2/\ldots/S_k)$ such that $\uplus_i S_i =[n]$. 
Note the use of parentheses to denote a sequence rather than a family of sets.
The set of these sequences is denoted $S^o([n],k)$.  We define the  inversion statistic exactly the same as for unordered signed partitions using \Cref{invB}, letting
\beq
\label{Inv:S^o}
\Inv\om = 
\{(s, S_j)\ :\ \text{$s\in S_i$ for some $i < j$ and $s \geq \min S_j$}\}
\eeq
and $\inv\om = \#\Inv\om$.
Using a similar proof to that of Theorem~\ref{SB:inv}, one can show the following.
\bth
\label{S^o:inv}
For $n,k\ge0$ we have

\vs{10pt}

\eqqed{S^o[n,k] =\sum_{\om\in S^o([n],k)} q^{\inv\om}.
}
\eth

\begin{defn}
We now define  the maps which will make up our involution in type $A$.  Given an ordered partition $\om=(S_1/S_2/\ldots/S_k)$, suppose $M=\max S_i$.
Say that $\om$ is {\em splittable at $M$} if $\#S_i\ge2$.  In that case the {\em splitting map} $\si_M$ is defined by
$$
\si_M(\om)=(S_1/\ldots/S_{i-1}/\{M\}/S_i-\{M\}/S_{i+1}/\ldots/S_k).
$$
We define $\om$ to be {\em mergeable at $M$} if
\begin{enumerate}
    \item $S_i=\{M\}$, and
    \item $M>\max S_{i+1}$.
\end{enumerate}
If these conditions hold then one can apply the {\em merging map} $\mu_M$ where
$$
\mu_M(\om) = (S_i/\ldots/S_{i-1}/S_i\uplus S_{i+1}/ S_{i+2}/\ldots/ S_k).
$$
\end{defn}

\begin{exa}
The ordered set partition $\om=(246/8/35/1/7)$ is splittable for $M=6$ and $5$, and
$$
\si_6(\om) = (6/24/8/35/1/7).
$$
On the other hand, $\om$ is only mergeable for $M=8$ and
$$
\mu_8(\om) = (246/358/1/7).
$$
\end{exa}
Note that if $\om$ is splittable at $M$, then $\si_M(\om)$ is mergeable at $M$ and $\mu_M\si_M(\om)=\om$.  The same statement holds with the roles of $\si_M$ and $\mu_M$ reversed.  Merge and split maps have been useful in a number of areas, including the computation of antipodes in Hopf algebras as shown by Benedetti and Sagan~\cite{BS:ai}.  We can now define the involution we will use for our first alternating sum.

\begin{defn}
\label{involA}
Define $\phi:\uplus_k S^o([n],k)\ra \uplus_k S^o([n],k)$ as follows.  Given $\om\in S^o([n],k)$,
find the largest $M$ (if any) such that $\om$ is either splitable or mergeable at $M=\max S_i$.  Note that, because of the restriction on $\#S_i$ for these two operations, it can not be both.  Let
$$
\phi(\om)=
\begin{cases}
\si_M(\om)  &\text{if $\om$ is splittable at $M$,}\\
\mu_M(\om)  &\text{if $\om$ is mergeable at $M$,}\\
\om         &\text{if no such $M$ exists.}
\end{cases}
$$
We see that $\phi$ is an involution because of the remarks at the end of the last paragraph and the fact that the largest splittable or mergeable $M$ is preserved by splitting or merging at this value.
\end{defn}

\bth
\label{alt:sumA}
For $n\ge0$, we have
$$
\sum_{k=0}^n (-q)^{n-k} S^o[n,k]=1.
$$
\eth
\begin{proof}
Define the sign of  $\om\in S^o([n],k)$ to be
\beq
\label{sgn:om}
\sgn \om=(-1)^{n-k}.
\eeq
The involution $\phi$ is sign-reversing on partitions which are not fixed since both $\si_M$ and $\mu_M$ change the number of blocks of $\om$ by $1$.  

By \Cref{S^o:inv}, we can write
\beq
\label{sgn:sum}
\sum_{k=0}^n (-q)^{n-k} S^o[n,k] 
=\sum_{\om\in\uplus_k S^o([n],k)} (\sgn\om) q^{n-k+\inv\om}.
\eeq
We claim that the terms for $\om$ and $\phi(\om)$ in equation~\eqref{sgn:sum} will cancel each other since  splitting adds one inversion and one block whereas merging removes one inversion and one block, hence $n-k+\inv$ is preserved by $\phi$.

To see that the claim holds when $\phi$ applies $\si_M$, note that
since $M$ is not the minimum of its block all the inversions of $\om$ will still be inversions of $\phi(\om)$.  Furthermore, if $M=\max S_i$ then splitting off $M$ will cause a new inversion $(M,S_i-\{M\})$.  Thus $\inv\phi(\om) =\inv\om + 1$.
Since $\phi$ is an involution, it also follows that
$\inv\phi(\om) =\inv\om - 1$. 

To complete the proof, it suffices to show that $\om_0=(1,2,\ldots,n)$ is a fixed point of $\phi$ and the only one since its contribution to~\eqref{sgn:sum} is $1$.  Clearly $\om_0$ is fixed since it has no blocks of size at least $2$ and its block maxima are in increasing order.  Conversely, if $\om$ is a fixed point, then it can have no blocks of size $2$ since then one could apply $\si_M$.  So $\om$ is a sequence of singleton blocks with increasing elements since otherwise $\mu_M$ could be applied.  The sequence $\om_0$ is the only ordered partition with these two properties.  This finishes the proof.
\end{proof}

To prove the type $B$ analogue of the previous theorem, we proceed in a similar manner.
An {\em ordered signed partition} of $\spn{n}$ is a sequence
$\om=(S_0/S_1/S_2/\ldots/S_{2k})$ satisfying conditions (1) and (2) in \Cref{B_part}.  Note that no assumption is made about standard form. 
The set of such partitions with $2k+1$ blocks is denoted $S_B^o(\spn{n},k)$.

The definition of inversion in \Cref{invB} remains unchanged.  But now it is possible to have an inversion where  $s=m_j$
if $m_j\in S_{j-1}$ and $-m_j\in S_j$.  The usual arguments give us the following result.
\bth
\label{S_B^o:inv}
For $n,k\ge0$ we have

\vs{10pt}

\eqqed{S_B^o[n,k] =\sum_{\om\in S_B^o(\spn{n},k)} q^{\inv\om}.
}
\eth

\begin{defn}
The {\em splitting} and {\em merging maps} have two cases in type $B$.  Consider $\om=(S_0/\ldots/S_{2k})$ and $M>0$ which is in a block with at least two elements.  
If $M=\max S_{2i-1}$ for some $i$, then $\si_M(\om)$ is the ordered signed partition formed by removing $M$ and $-M$ from their blocks and adding a block pair $-M/M$ immediately to the left of what remains of $S_{2i-1}$.
If $M=\max S_{2i}$ for some $i$, then $\si_M(\om)$ is  obtained by removing $M$ and $-M$ from their blocks (which will be the same if $i=0$) and adding a block pair $M/{-M}$ immediately to the right of the remains of $S_{2i}$.

Now suppose that $M$ is in a singleton block, which implies that the same is true of $-M$. If the block pair is $S_{2i-1}/S_{2i}=-M/M$ and $M>\max|S_{2i+1}|$ then add $M$ to $S_{2i+1}$ and $-M$ to $S_{2i+2}$ to form $\mu_M(\om)$.
If the block pair is $S_{2i-1}/S_{2i}=M/{-M}$ and $M>\max|S_{2i-2}|$ then $\mu_M(\om)$ is obtained by adding $M$ to $S_{2i-2}$ and $-M$ to the same block if $2i-2=0$ or to the block to its left otherwise.
\end{defn}

\begin{exa}
Here are examples of splitting (the forward arrows) and merging (the reverse arrows) to illustrate all of the possible cases.  
\begin{align*}
  (\overline{4}\overline{1}014 \mid 2\overline{3}/\overline{2}3)
    &\stackrel{M=4}{\longleftrightarrow}
    (\overline{1}01 \mid 4/\overline{4} \mid 2\overline{3}/\overline{2}3), \\
  (\overline{5}\overline{4}045 \mid \overline{2}3\overline{6}/2\overline{3}6 \mid 1/\overline{1})
    &\stackrel{M=6}{\longleftrightarrow}
    (\overline{5}\overline{4}045 \mid \overline{2}3/2\ol{3} \mid 6/\ol{6} \mid 1/\ol{1}), \\
  (\overline{5}\overline{4}045 \mid \overline{2}36/2\overline{3}\overline{6} \mid 1/\overline{1})
    &\stackrel{M=6}{\longleftrightarrow} 
    (\ol{5}\ol{4}045 \mid \ol{6}/6 \mid \overline{2}3/2\overline{3} \mid 1/\overline{1}). \\
\end{align*}
\end{exa}

The map $\phi:\uplus_k S_B^o(\spn{n},k)\ra \uplus_k S_B^o(\spn{n},k)$ is defined exactly as in \Cref{involA} for the type $A$ case, merely substituting the signed splitting and merging maps.  As in the previous case, it is easy to see that $\phi$ is an involution.

\bth
\label{alt:sumB}
For $n\ge0$ we have

\vs{10pt}

\eqqed{
\sum_{k=0}^n (-q)^{n-k} S^o_B[n,k]=1.
}
\eth
\bprf
A sign is assigned to $\om\in S_B^o(\spn{n},k)$ using \eqref{sgn:om} again.  Now the proof continues in much the same manner as that of \Cref{alt:sumA} using the previous theorem in place of Theorem~\ref{S^o:inv} and with unique fixed point
$\om_0=(0,\ol{1},1,\ol{2},2,\ldots,\ol{n},n)$.
\eprf

Note that Theorems~\ref{alt:sumA} and~\ref{alt:sumB}  can be given algebraic  proofs by setting $t=1$ and  then substituting $x_i = -q[i-1]$ or $x_i=-q[2i-1]$, respectively, in \Cref{t^n:x}.

\subsection{Divisibility}

The following two results are about divisibility in the ring $\bbZ[q]$.
They are analogues of \Cref{alt:sumA} and \Cref{alt:sumB} for larger powers of $q$.  Algebraically, the results here follow immediately from the corresponding ones in the previous subsection by using the fact that
$$
q^{m(n-k)}\Cong\ q^{n-k}\ (\Mod q^m-q)
$$
for $m\ge1$ and $n>k$.
We will also show how they can be proved combinatorially using the sign-reversing involutions already developed.  
\begin{thm}
\label{div:sumA}
For $m\ge1$ and $n\ge0$ we have
$$
\sum_{k=0}^n (-1)^{n-k} q^{m(n-k)} S^o[n,k]\Cong 1\ (\Mod q^m-q).
$$
\end{thm}
\begin{proof}
Let $\Si$ be the sum under consideration.  Then clearly $\Si$ has constant term $1$ so that $\Si-1$ is divisible by $q$.  We must also show that it is divisible by $q^{m-1}-1$.

By \Cref{S^o:inv}, we can write
$$
\Si 
=\sum_{\om\in\uplus_k S^o([n],k)} (\sgn\om) q^{m(n-k)+\inv\om}
$$
where the sign is given by~\eqref{sgn:om}.
Recall that $\om=(1,2,\ldots,n)$ is the only fixed point of the involution $\phi$ on $S^o([n],k)$ and that its contribution to the previous sum is $1$.
So we need only modify the demonstration   of \Cref{alt:sumA}  by showing that, for non-fixed points $\om\in S^o([n],k)$ of $\phi$, the sum of the contributions of $\om$ and $\phi(\om)$ is divisible by $q^{m-1}-1$.

We will just give details when $\phi$ applies $\mu_M$.  But then $\phi(\om)$ has $k-1$ blocks and, as proved in the proof  of \Cref{alt:sumA}, $\inv\phi(\om)=\inv\om-1$.  So, up to sign, the contribution of these two ordered set partitions is
$$
q^{m(n-k)+\inv\om} - q^{m(n-k+1)+\inv\om-1}=
(1-q^{m-1})q^{m(n-k)+\inv\om}
$$
as desired.
\end{proof}

The type $B$ analogue of the previous result is obtained by modifying the proof of \Cref{alt:sumB} similarly to how we just modified the demonstration of \Cref{alt:sumA}.  So we omit the details.
\bth
\label{div:sumB}
For $n\ge0$ we have

\vs{10pt}

\eqqed{
\sum_{k=0}^n (-1)^{n-k}q^{m(n-k)} S^o_B[n,k] \Cong 1\ (\Mod q^m-q).
}
\eth

\section{Coinvariant algebras}
\label{ca}

In this section we will propose analogues of the Artin basis for certain super coinvariant algebras in types $A$ and $B$.  If these sets can be shown to be bases, then it will follow that the corresponding bigraded Hilbert series can be expressed in terms of ordered $q$-Stirling numbers. Since our bases are new even in type $A$, we will deal with that case first and then move on to type $B$.

\subsection{Type $A$ coinvariants}
\label{tA}

Consider the {\em $k$th power sum symmetric polynomial}
$$
p_k(n) = x_1^k + x_2^k +\cdots + x_n^k.
$$

\begin{defn}
The {\em type $A$ coinvariant algebra} is the finite-dimensional  commutative algebra
$$
\R_n = \frac{\bbQ[x_1,\ldots,x_n]}{\spn{p_k(n)\ :\ k\in[n]}}
$$
where $\bbQ$ is the rational numbers.  This algebra is graded by degree and we let $(\R_n)_d$ denote the $d$th graded piece.  We will not make a distinction in our notation between a polynomial in $\bbQ[x_1,\ldots,x_n]$ and its representative in $\R_n$.
\end{defn}

There is a standard basis for $\R_n$.  We will use it as a model for our bases in the super coinvariant algebras we consider.

\begin{defn}
The {\em Artin basis} for $\R_n$ is
$$
\cA_n =\{x_1^{m_1} x_2^{m_2} \cdots x_n^{m_n}\ :\ \text{$0\le m_i\le i-1$ for $i\in[n]$}\}.
$$
\end{defn}

The next result follows immediately from the fact that $\cA_n$ is a basis for $\R_n$. See \Cref{SW:hdf} for further history and details.
\bth
The coinvariant algebra $\R_n$ has Hilbert series

\vs{10pt}

\eqqed{
\Hilb(\R_n;q):=\sum_{d\ge0} \dim (\R_n)_d\ q^d = [n]!.
}
\eth

There is an alternative description of $\cA_n$ in terms of compositions which will be useful in the sequel.  A {\em weak composition of $d$ with $n$ parts} is a sequence of nonnegative integers $\al=(\al_1,\ldots,\al_n)$ where $|\al|:=\sum_i \al_i = d$.  The {\em  diagram} of $\al$ consists of $n$  columns lying on the same line with $\al_i$ boxes in column $i$ for $i\in[n]$.  See the diagram on the left in \Cref{stair} for an example. 
We will also use the partial order on compositions with $n$ parts given by $\al\le\be$ if $\al_i\le\be_i$ for all $i\in[n]$, equivalently, if the diagram of $\al$ is contained in the diagram of $\be$.  This relation is also illustrated in \Cref{stair}. Every composition $\al=(\al_1,\ldots,\al_n)$ has an associated monomial
$$
\bx^\al = x_1^{\al_1} x_2^{\al_2}\cdots x_n^{\al_n}
$$
of degree $d=|\al|$.  The Artin basis can be described as
$$
\cA_n=\{ \bx^\al\ :\ \al\le(0,1,\ldots,n-1)\}.
$$
We call $(0,1,\ldots,n-1)$ the {\em staircase}.

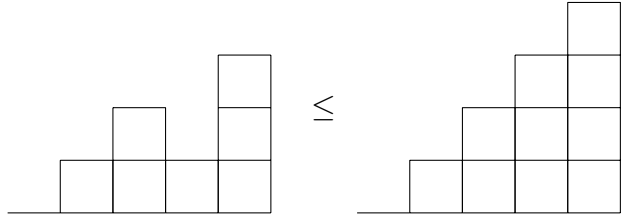
\begin{figure}
\bce
\begin{tikzpicture}[scale=.7]
\draw (0,0)--(1,0);
\draw (1,0) grid (2,1);
\draw (2,0) grid (3,2);
\draw (3,0) grid (4,1);
\draw (4,0) grid (5,3);
\draw (6,2) node{$\le$};
\end{tikzpicture}
\begin{tikzpicture}[scale=.7]
\draw (0,0)--(1,0);
\draw (1,0) grid (2,1);
\draw (2,0) grid (3,2);
\draw (3,0) grid (4,3);
\draw (4,0) grid (5,4);
\end{tikzpicture}
\ece
    \caption{The diagram of $\al=(0,1,2,1,3)$ contained in the staircase $(0,1,2,3,4)$}
    \label{stair}
\end{figure}

There is a third description of $\cA_n$ involving permutations in the symmetric group $\fS_n$.  As usual, an {\em inversion} of a permutation $\pi=\pi_1\ldots\pi_n$ written in one-line notation is a pair $(i,j)$ with $i<j$ and $\pi_i>\pi_j$.  Let
$$
\inv_i \pi = \#\{j\ :\ \text{$(i,j)$ is an inversion of $\pi$}\}.
$$
The {\em inversion composition} of $\pi\in\fS_n$ is
$$
I(\pi) = (\inv_1(\pi),\inv_2(\pi),\ldots,\inv_n(\pi)).
$$
Clearly $I(\pi)\le (0,1,\ldots,n-1)$.  In fact, it is well-known and easy to prove that the map
$$
I:\fS_n \ra \{\al\ :\ \al\le (0,1,\ldots,n-1)\}
$$
is a bijection.  It follows that
$$
\cA_n = \{\bx^{I(\pi)}\ :\ \pi\in\fS_n\}.
$$

\subsection{Type $A$ super coinvariants}\label{sco_A}

We now turn our attention to super coinvariant algebras.  Let $\th_1,\ldots,\th_n$ be anticommuting variables so that
$$
\th_i \th_j = -\th_j \th_i
$$
for all $i,j\in[n]$.  Note that because of anticommutivity we have 
\beq
\label{th^2}
\th_i^2=0
\eeq
for all $i\in[n]$.  We also assume that the $\th_i$ and $x_j$ commute with each other.   Define the {\em $k$th super power sum polynomial} to be
$$
sp_k(n) = x_1^k\th_1 + x_2^k\th_2 + \cdots + x_n^k\th_n.
$$

\begin{defn}
The {\em type $A$ super coinvariant algebra} is the finite-dimensional  algebra
$$
\SR_n = \frac{\bbQ[x_1,\ldots,x_n,\th_1,\ldots,\th_n]}{\spn{p_k(n),\ sp_{k-1}(n)\ :\ k\in[n]}}.
$$
This algebra is bi-graded where we let $(\SR_n)_{d,e}$ denote the graded piece
with monomials which are of degree $d$ is the $x$'s and degree $e$ in the $\th$'s. 
\end{defn}

Zabrocki~\cite{zab:mdc} has conjectured a description for the tri-graded Frobenius characteristic of the super-diagonal coinvariant algebra of $\fS_n$ involving two sets of commuting and one set of anti-commuting variables. Specialized to $\SR_n$, it becomes the following which explains our interest in this algebra.
\bcon
\label{SR_n:con}
We have
$$
\Hilb(\SR_n;q,t):=\sum_{d,e\ge0} \dim (\SR_n)_{d,e}\ q^d t^e 
= \sum_{k\ge0} S^o[n,k] t^{n-k}.
$$
\econ

We will now propose an analogue of the Artin basis for $\SR_n$.  Note that by~\eqref{th^2}, nonzero $\th$ monomials must contain at most one copy of each $\th_i$.  So such monomials are indexed by subsets $T\sbe[n]$ and we let
$$
\th_T = \th_{t_1}\th_{t_2}\cdots \th_{t_k}
$$
where $T=\{ t_1<t_2<\cdots<t_k\}$.  Letting
$$
[a,b]=\{a,a+1,\ldots,b\}
$$
for $a,b\in\bbZ$, we will only need $T\sbe[2,n]$ for our proposed basis.  Given such a subset we define the {\em $\al$-sequence of $T$}, $\al(T)$, to be the composition constructed recursively by letting $\al_1(T)=0$ and, for $i\in[2,n]$,
$$
\al_i(T) = \al_{i-1}(T) + \case{0}{if $i \in T$,}{1}{if $i \not\in T$.} 
$$
The diagram of an example of $\al(T)$ will be found in Figure~\ref{al(T):fig}.

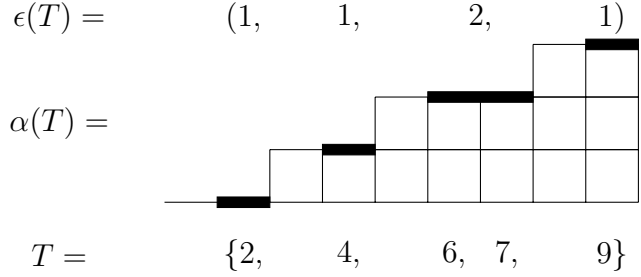
\begin{figure}
\bce
\begin{tikzpicture}[scale=.7]
\draw(-2,1.5) node{$\al(T)=$};
\draw(-2,3.5) node{$\ep(T)=$};
\draw(1.5,3.5) node{$(1,$};
\draw(3.5,3.5) node{$1,$};
\draw(6,3.5) node{$2,$};
\draw(8.5,3.5) node{$1)$};
\draw (0,0)--(2,0);
\draw[fill] (1,-.1) rectangle (2,.1);
\draw (2,0) grid (9,1);
\draw[fill] (3,.9) rectangle (4,1.1);
\draw (4,1) grid (9,2);
\draw[fill] (5,1.9) rectangle (7,2.1);
\draw (7,2) grid (9,3);
\draw[fill] (8,2.9) rectangle (9,3.1);
\draw(-2,-1) node{$T=$};
\draw(1.5,-1) node{$\{2,$};
\draw(3.5,-1) node{$4,$};
\draw(5.5,-1) node{$6,$};
\draw(6.5,-1) node{$7,$};
\draw(8.5,-1) node{$9\}$};
\end{tikzpicture}
\ece
    \caption{The compositions $\al(T)$ and $\ep(T)$ when $n=9$ and $T=\{2,4,6,7,9\}$}
    \label{al(T):fig}
\end{figure}

\begin{defn}
\label{SA_n:def}
The {\em super Artin set} for $\SR_n$ is
$$
\cS\cA_n =\{x^\al \th_T\ :\ \text{$T\sbe[2,n]$ and $\al\le\al(T)$}\}.
$$
\end{defn}

A different description of this set was independently discovered by a group of mathematicians associated with the Fields Institute including Nantel Bergeron, Shu Xiao Li, John Machachek, Robin Sulzgr\"uber, and Mike Zabrocki~\cite{zab:pc}.

\bcon
\label{SA_n:con}
The set $\cS\cA_n$ is a basis for $\SR_n$.
\econ
This conjecture has been verified using Macaulay2~\cite{GS:M2} for $n \leq 6$. While the classical Artin basis is the standard monomial basis for any monomial order with $x_1 > x_2 > \cdots > x_n$ (see \cite[\S5]{SW:hdf}), calculations by the Fields Institute group  have shown that  under reasonable assumptions
the same cannot be true of $\cS\cA_n$.

It will be convenient to think of $\al(T)$ as an elongated version of the staircase.  For example, the composition in \Cref{al(T):fig} is the staircase $(0,1,2,3)$ stretched out by adding a column of length $0$, a column of length $1$, two columns of length $2$, and a column of length $3$.  Formally, suppose $T\sbe[2,n]$ with $\#T=n-k$.  We then define the associated {\em expansion composition}
$\ep(T) = (\ep_1(T),\ldots,\ep_k(T))$ where
$$
\ep_j(T) = \#\{i\in T\ :\ \al_i(T) = j-1\}.
$$
Returning to our example, we have $\ep(T)=(1,1,2,1)$ corresponding to the bold line segments in the diagram for $\al(T)$.  Note that $|\ep(T)|=n-k$.
\bpr
\Cref{SA_n:con} implies \Cref{SR_n:con}.
\epr
\begin{proof}
Assume that $\cS\cA_n$ is a basis for $\SR_n$.  For fixed $T\sbe[2,n]$ with $\#T=n-k$, the description of the expansion composition $\ep(T)=(\ep_1,\ep_2,\ldots,\ep_k)$ shows that the monomials in $\cS\cA_n$ whose theta component is $\th_T$ contribute
$$
[1]^{\ep_1+1}[2]^{\ep_2+1}\cdots [k]^{\ep_k+1}
=[k]! [1]^{\ep_1}[2]^{\ep_2}\cdots [k]^{\ep_k}
$$
to the $q$-grading in $\Hilb(\SR_n;q,t)$.  Summing over all such $T$ gives a contribution of
$$
[k]! h_{n-k}([1],[2],\ldots,[k])=S^o[n,k].
$$
Since $\th_T$ has degree $t^{n-k}$ for these $T$, the proof is complete.
\end{proof}

There is also a way to express the elements of $\cSA_n$ by using inversions in type $A$ ordered set partitions $\om=(S_1/S_2/\ldots)$ of $[n]$, for which we write $\om\comp[n]$.  Given $s\in[n]$ we define
\beq
\label{inv_s}
\inv_s\om = \#\{S_j\ :\ (s,S_j)\in \Inv\om\}
\eeq
where $\Inv\om$ is defined by~\eqref{Inv:S^o}.  From this we get the {\em inversion composition}
\beq
\label{eta}
\eta(\om)=(\inv_1\om,\inv_2\om,\ldots,\inv_n\om).
\eeq
We also need the set
\beq
\label{T(om)}
T(\om) 
= \{ t\in[n]\ :\ \text{$t\in S_i$ for some $i$ and $t>m_i$}\}
\eeq
where $m_i=\min|S_i|$.  Note that the absolute value is not needed here since all elements of $S_i$ are positive.
Also, it is impossible for $1\in T(\om)$.  But this description will permit us to use the exactly the same definition in type $B$ where the absolute value is needed and $1\in T(\om)$ is possible.
For example, if $\om=(S_1/S_2/S_3)=(25/136/4)$ then $\inv_2\om = 1$ because of $S_2$, $\inv_5\om=2$ because of $S_2$ and $S_3$,
$\inv_6\om=1$ because of $S_3$, and $\inv_s\om=0$ for all other values of $s$, so that $\eta(\om)=(0,1,0,0,2,1)$.  Furthermore, $T(\om)=\{3,5,6\}$.
\bpr
We have
$$
\cSA_n =\{\bx^{\eta(\om)} \th_{T(\om)}\ :\ \om\comp[n]\}.
$$
\epr
\bprf
It suffices to define a weight-preserving bijection  from the pairs $(T,\al)$ appearing in~\Cref{SA_n:def} to the $\om\comp[n]$.  Given $(T,\al)$ we construct $\om$ inductively as follows.  We start with $\om=(1)$ and insert the numbers $2,3,\ldots,n$ in order according to the following rules when it comes to inserting $k$.
\ben
\item  If $k\in T$ then put $k$ in the existing block of $\om$ so that  exactly $\al_k$ new inversions result.
\item If $k\not\in T$ then make $k$ a new block of $\om$ so that exactly $\al_k$ new inversions result.
\een
It is routine to verify that this is a well defined map and to describe its inverse, so those details are left to the reader.
\eprf

\begin{exa}
The reader will note how the proof just given mirrors the standard combinatorial demonstration that
$$
\sum_{\pi\in\fS_n} q^{\inv\pi} = [q]!.
$$
To illustrate the construction in the proof of this result suppose that $n=5$, $T=\{3,5\}$, and $\al=(0,1,0,2,1)$.  The sequence of ordered partitions constructed is
$$
(1),\ (2/1),\ (2/13),\ (4/2/13),\ (4/25/13).
$$
For example, when $4$ is inserted  into $(2/13)$, then, since $4\not\in T$, it will appear as a singleton block.  And since $\al_4=2$ it must be the first block to cause two inversions.  Similarly, when $5$ is inserted then $5\in T$ forces this element into one of the existing blocks.  And if $5$ is to cause one new inversion then it must be in the second block from the right. The $\alpha$-sequence of $T$ is simply the sequence of the maximal number of inversions one could possibly cause at each step. 
\end{exa}

Our $\inv$ statistic, or equivalently Steingr\'imsson's $\operatorname{ros}$ \cite{ste:sop}, effectively numbers the possible insertion positions ``from right to left'' starting at $0$. One may get equidistributed variations on the $\inv$ statistic by changing this numbering scheme. Using the left-to-right order yields Steingr\'imsson's $\operatorname{los}$ \cite{ste:sop}, or equivalently (in the unordered case) Cai--Readdy's $\operatorname{wgt}$ \cite{CR:qStir} which is Wachs--White's $\operatorname{ls}$ up to a $q$-shift \cite{WW:pqs}.

\subsection{Type $B$ super coinvariants}
\label{tB}

We now consider coinvariant algebras in type $B$.
\begin{defn}
The {\em type $B$ coinvariant algebra} is the finite-dimensional, graded,  commutative algebra
$$
\RB_n = \frac{\bbQ[x_1,\ldots,x_n]}{\spn{p_{2k}(n)\ :\ k\in[n]}}.
$$
\end{defn}

The analogue of the Artin basis in this context is as follows.
\begin{defn}
The {\em type $B$ Artin basis} for $\RB_n$ is
$$
\cB_n =\{\bx^\al\ :\ \al\le(1,3,\ldots,2n-1)\}.
$$
\end{defn}

\begin{figure}
\bce
\begin{tikzpicture}[scale=.7]
\draw (0,0) grid (1,1);
\draw (1,0) grid (2,3);
\draw (2,0) grid (3,5);
\draw (3,0) grid (4,7);
\end{tikzpicture}
\ece
    \caption{The double staircase $(1,3,5,7)$}
    \label{dstair}
\end{figure}
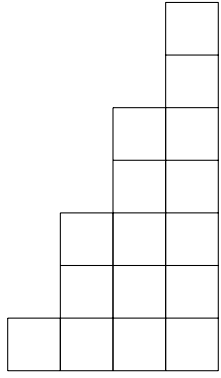

We call the composition $(1,3,\ldots,2n-1)$ the {\em double staircase} and it is displayed in \Cref{dstair} for $n=4$.
Again, the Artin basis trivializes the computation of the Hilbert series.
\bth
The coinvariant algebra $\RB_n$ has Hilbert series

\vs{10pt}

\eqqed{
\Hilb(\RB_n;q)= [2n]!!.
}
\eth

Swanson and Wallach~\cite{SW:hdf} considered the type $B$ super coinvariant algebra where one adds anticommuting variables $\th_1,\ldots,\th_n$ which again commute with the $x_j$'s.
\begin{defn}
The {\em type $B$ super coinvariant algebra} is the finite-dimensional, bigraded,  algebra
$$
\SRB_n = \frac{\bbQ[x_1,\ldots,x_n]}{\spn{p_{2k}(n),\ sp_{2k-1}(n)\ :\ k\in[n]}}.
$$
\end{defn}

As for the super coinvariant algebra in type $A$, the Hilbert series is only conjectural.
\bcon[\cite{SW:hdf}]
\label{SRB_n:con}
We have
$$
\Hilb(\SRB_n;q,t) = \sum_{k\ge0} S_B^o[n,k] t^{n-k}.
$$
\econ

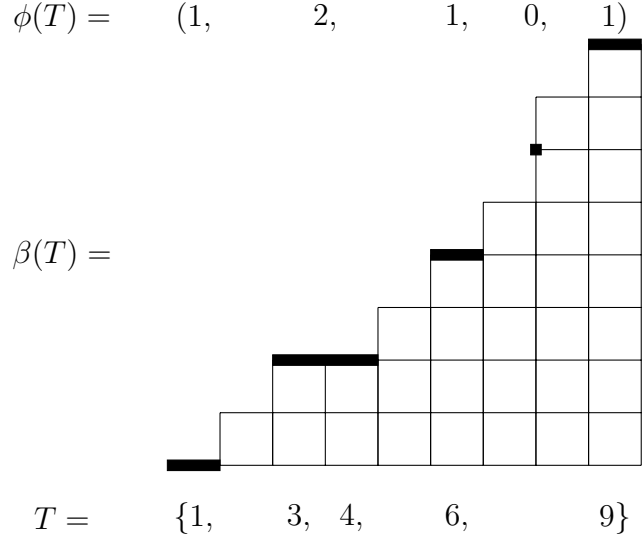
\begin{figure}
\bce
\begin{tikzpicture}[scale=.7]
\draw(-2,4) node{$\be(T)=$};
\draw(-2,8.5) node{$\phi(T)=$};
\draw(.5,8.5) node{$(1,$};
\draw(3,8.5) node{$2,$};
\draw(5.5,8.5) node{$1,$};
\draw(7,8.5) node{$0,$};
\draw(8.5,8.5) node{$1)$};
\draw[fill] (0,-.1) rectangle (1,.1);
\draw (1,0) grid (2,1);
\draw (2,0) grid (4,2);
\draw[fill] (2,1.9) rectangle (4,2.1);
\draw (4,0) grid (5,3);
\draw (5,0) grid (6,4);
\draw[fill] (5,3.9) rectangle (6,4.1);
\draw (6,0) grid (7,5);
\draw (7,0) grid (8,7);
\draw[fill] (6.9,5.9) rectangle (7.1,6.1);
\draw (8,0) grid (9,8);
\draw[fill] (8,7.9) rectangle (9,8.1);
\draw(-2,-1) node{$T=$};
\draw(.5,-1) node{$\{1,$};
\draw(2.5,-1) node{$3,$};
\draw(3.5,-1) node{$4,$};
\draw(5.5,-1) node{$6,$};
\draw(8.5,-1) node{$9\}$};
\end{tikzpicture}
\ece
    \caption{The compositions $\be(T)$ and $\phi(T)$ when $n=9$ and $T=\{1,3,4,6,9\}$}
    \label{be(T):fig}
\end{figure}

We have a set of elements of $\SRB_n$ which, if they form a basis, will verify the previous conjecture.  
To define the analogue of the $\al$-sequence, it will be convenient to use the notation $\chi(\cS)$ which is $1$ if the statement $\cS$ is true, or $0$ if it is false.
Let the {\em $\be$-sequence} of $T\sbe[n]$ to be the composition defined recursively by 
$\be_1(T)=\chi(1\not\in T)$ and
$$
\be_i(T) = \be_{i-1}(T)+\chi(i\not\in T)+ \chi(i-1\not\in T)
$$
for $i\in[2,n]$.  \Cref{be(T):fig} contains an example.

\begin{defn}
\label{SAB_n:def}
The {\em super Artin set} for $\SRB_n$ is
$$
\cSAB_n =\{x^\al \th_T\ :\ \text{$T\sbe[n]$ and $\al\le\be(T)$}\}.
$$
\end{defn}

We conjecture that $\cSAB$ is, in fact, a basis.
\bcon
\label{SAB_n:con}
The set $\cSAB_n$ is a basis for $\SRB_n$.
\econ

Similar to type $A$, the composition $\be(T)$ can be considered as an expansion of the double staircase.  Let $T\sbe[n]$ with $\#T=n-k$.  The {\em type $B$ expansion composition} is $\phi(T)=(\phi_0(T),\ldots,\phi_k(T))$ where
$$
\phi_j(T) = \#\{i\in T\ :\ \be_i(T) = 2j\}.
$$
\Cref{be(T):fig} also lists the expansion composition for the given $n$ and $T$.  So $\phi_j(T)$ is just the number of columns of $\be(T)$ of height $2j$.  It is not hard to see that removing all the even height columns from $\be(T)$ leaves a copy of the double staircase 
$(1,3,\ldots,2k-1)$ and thus $|\phi(T)|=n-k$.
\bpr
\Cref{SAB_n:con} implies \Cref{SRB_n:con}.
\epr
\begin{proof}
Suppose $T\sbe[n]$ with $\#T=n-k$.  The discussion of 
$\phi(T)=(\phi_0,\ldots,\phi_k)$ just given 
shows that the monomials in $\cS\cA\cB_n$  whose theta component is $\th_T$ have a factor of $[2k]!!$ from the columns corresponding to the double staircase, and a factor of 
$[1]^{\phi_0}[3]^{\phi_1}\cdots[2k+1]^{\phi_k}$ from the column of even length.  Summing over all such $T$ gives a contribution of
$$
[2k]!! h_{n-k}([1],[3],\ldots,[2k+1])=S_B^o[n,k].
$$
to the $q$-grading.
Now the fact that  $\th_T$ has degree $t^{n-k}$ for these $T$ completes the proof. 
\end{proof}

For the  description of $\cSAB$ in terms of ordered set partitions $\om$ we will use the same notation as in type $A$. This will cause no confusion because it will be clear from context whether the necessary functions are being applied to a partition which is type $A$ or type $B$.  We write $\om\comp\spn{n}$ if $\om=(S_0/S_1/S_2/\ldots)$ is an ordered set partition of $\spn{n}$.  The definitions~\eqref{inv_s}, \eqref{eta}, and \eqref{T(om)} carry over to type $B$ without change.  For example, if we have
$\om=(0\ol{1}1 \mid 4/\ol{4} \mid \ol{2}3/2\ol{3})$ then 
$\eta(\om)=(0,0,1,3)$ and $T(\om)=\{1,3\}$.

\bpr
We have
$$
\cSAB_n =\{\bx^{\eta(\om)} \th_{T(\om)}\ :\ \om\comp\spn{n}\}.
$$
\epr
\bprf
As in type $A$, 
it suffices to define a weight-preserving bijection  from the pairs $(T,\al)$ appearing in~\Cref{SAB_n:def} to the $\om\comp\spn{n}$.  Given $(T,\al)$ we construct $\om$ by starting with $\om=(0)$ and inserting the numbers $\pm1,\pm2,\ldots,\pm n$ in order according to the following rules.  Note that the position of $k$ forces the position of $-k$ so that either both are in the zero block or $-k$ is in the block paired to the one containing $k$.
\ben
\item  If $k\in T$ then put $k$ in the existing block of $\om$ so that  exactly $\al_k$ new inversions result.  (The forced insertion of $-k$ will not cause any new inversions in this case.)
\item If $k\not\in T$ then make $k$ and $-k$ a new pair of blocks of $\om$ so that exactly $\al_k$ new inversions result.
\een
As before, details that this map is well defined and invertible are straightforward and so omitted.
\eprf

\begin{exa}
Suppose that $n=3$, $T=\{2,3\}$, and $\al=(0,1,2)$.  We start with $\om=(0)$.  Since $1\not\in T$, we must add $1$ and $\ol{1}$ as separate blocks.  And $\al_1=0$  means that no inversions are to be created so that now 
$$
\om=(0\mid \ol{1}/1).
$$
We have $2\in T$ so it must be placed in a block with other elements.  And it must create $\al_2=1$ new inversion.  This forces $2$ into the block with $\ol{1}$ and also $\ol{2}$ into the block with $1$ resulting in
$$
\om=(0\mid \ol{1}2/1\ol{2}).
$$
Finally, $3\in T$ so $T$ will go into one of the existing blocks.  It must create $\al_3=2$ inversions and so must reside in the zero block.  It follows that $\ol{3}$ is also in $S_0$ and we finally have
$$
\om=(0\ol{3}3\mid \ol{1}2/1\ol{2}).
$$
\end{exa}

\section{Comments and open questions}
\label{coq}

In this section we collect various comments and open questions raised by the present work.

\subsection{Complex reflection groups}

If $G$ is any complex reflection group, then one can define Stirling numbers of the first and second kind for $G$ using the Whitney numbers of the first and second kind, respectively, for $G$'s intersection lattice analogous to \Cref{Pi_B_n} in type $B$.
It follows from the work of Shephard and Todd~\cite{ST:fur} that the Stirling numbers of the first kind can be expressed as elementary symmetric polynomials in the coexponents of $G$.  For the Stirling numbers of the second kind, the situation is more complicated and this approach is explored in~\cite{SS:snc}.

\subsection{Major index for signed permutations and super coinvariant bases}

We showed in \Cref{B_part} that $S_B[n,k]$ can be viewed as the generating function for both an inversion and a major index statistic on signed partitions.  By contrast, only an inversion generating function was given for $c_B[n,k]$ in \Cref{cB:inv}.  It would be interesting to find a major index analogue for permutations in type $B$.  

In another direction, the Artin basis in type $A$ has a well-known ``major index analogue,'' the Garsia--Stanton basis 
$$
\left\{\prod_{i \in \Des(\pi)} x_{\pi(1)} \cdots x_{\pi(i)}\ 
:\ \sigma \in S_n\right\}.
$$
See~\cite{GS:SR}.
Adin--Brenti--Roichman \cite{ABR:drms} gave a type $B$ analogue of the Garsia--Stanton basis using the \textit{flag-major index} of Adin--Roichman \cite{AR:fmaj}. It would be interesting to give conjectured super coinvariant extensions of these bases in types $A$ and $B$. 

\subsection{$q$-difference equations}
\label{qec}

The generating function 
$$
C = \sum_{n, k \geq 0} c[n, k] t^k \frac{x^n}{[n]!}
$$ 
in
\Cref{cnk_egf} satisfies the first-order linear $q$-difference equation
\beq
\label{eq_C_one}
  D_q C - \frac{tC}{1-x} = 0
\eeq
where $C(0, t) = 1$.
The classical solution of $y' + p(x)y = 0$ given by $y = \exp\left(-\int p(x)\,dx\right)$ generalizes to $q$-difference equations using Gessel's $q$-composition and $q$-chain rule. In particular, if $D_q Y + P(x) Y = 0$, then $Y = \exp\left[\int P(x)\,d_qx\right]$, where we have used the $q$-antiderivative. From \eqref{eq_C_one}, $C = \exp[-y\log_q(1-x)]$.

Now consider 
$$C_B = \sum_{n, k \geq 0} c_B[n, k] t^k \frac{x^n}{[n]!}.
$$
Using manipulations similar to those in the proof of \Cref{cnk_egf}, one may show $C_B$ obeys the following second-order linear $q$-difference equation.
\begin{lem}
  We have
  \beq
    \label{eq_CB_two}
    x^2 q^2(1-q)D_q^2 C_B + (1-q(1+q)x) D_q C_B - (1+t) C_B = 0 
  \eeq
where $C_B(0, t) = 1$.
\end{lem}
\noindent When $q=1$, this becomes a first order differential equation with solution $(1-2x)^{-(1+t)/2}$, giving an alternate proof of \Cref{egf:thm}(d).

We are unaware of techniques which allow us to solve general second-order linear q-differential equations in terms of well-known operations. However, we may rewrite \eqref{eq_CB_two} as a first order equation using the \textit{$q$-pre-composition operator} $Q(F(x)) = F(qx)$:
  \[ D_q C_B - qx D_q (1+Q) C_B - (1+t) C_B = 0. \]
If we replace $Q$ with $q$ (i.e. post-multiply by $q$ rather than pre-multiply by $q$), we may solve the resulting $q$-differential equation using the methods of the previous paragraph, which may be considered as a ``first approximation'' of $C_B$. More explicitly, we have
  \[ (1 - q(1+q) x) D_q  \widetilde{C}_B - (1+t)  \widetilde{C}_B = 0 \]
where
  \[ \widetilde{C}_B = \exp_q\left[-\frac{1+t}{q(1+q)} \log_q[1 - q(1+q) x]\right]. \]
When $q=1$, we again recover \Cref{egf:thm}(d). Since a wide variety of special functions satisfy second order linear differential equations, solving \eqref{eq_CB_two} in a recognizable way would be interesting.

\subsection{Sign-reversing involutions}

In \Cref{as} we gave combinatorial proofs of the alternating sums involving $S^o[n,k]$ and $S_B^o[n,k]$ using sign-reversing involutions.  We also showed how these equations were special cases of the symmetric function function identity in \Cref{t^n:x}, where that result was demonstrated by algebraic manipulation.  Can \Cref{t^n:x} itself be proved by sign-reversing involution?

\subsection{Log concavity}

Partially order $\bbR[\bx]$ by letting $f(\bx)\le g(\bx)$ if $g(\bx)-f(\bx)\in\bbR^+[\bx]$ where $\bbR^+$ is the nonnegative reals.  A sequence of polynomials $(f_k(\bx))_{k\ge0}=f_0(\bx), f_1(\bx),f_2(\bx),\ldots$ is said to be {\em $\bx$-log-concave} if
$$
f_k(\bx)^2 \ge f_{k-1}(\bx) f_{k+1}(\bx)
$$
for all $k>0$.  Call the sequence {\em strongly $\bx$-log-concave} if
$$
f_k(\bx) f_l(\bx) \ge f_{k-1}(\bx) f_{l+1}(\bx)
$$
for all $l\ge k>0$.  Clearly strong $\bx$-log concavity implies $\bx$-log concavity.  The converse is not true in general, although it is well known that it does hold if the polynomials are all positive constants. And in the case of a sequence of constants we say that it is just 
{\em log-concave}.

The following is a corollary of Theorem 2.6 in~\cite{sag:lcs}.
\bth
If the sequence $x_1,x_2,x_3,\dots$ is strongly $\bx$-log-concave then so are the following sequences
$$
(e_k(n))_{n\ge0} \qmq{and} (h_k(n))_{n\ge0}
$$
where $k$ is fixed, as well as
$$
(e_{k-j}(n+j))_{j\ge0} \qmq{and} (e_{k-j}(n+j))_{j\ge0}
$$
where both $k$ and $n$ are fixed.\hqed
\eth
All the definitions in the previous paragraph apply when there is a single variable $q$.
It is easy to see that the sequence $[1]_q,[2]_q,[3]_q,\ldots$ is strongly $q$-log-concave.  So any subsequence will be as well.  Combining this fact, the previous theorem, and parts (a) and (b) of \Cref{cS:eh} give the following result.
\bco
For fixed $n$, the sequences $(S_B[n,k])_{k\ge0}$ and $(c_B[n,k])_{k\ge0}$ are strongly $q$-log-concave.\hqed
\eco

A condition that is implied by log concavity for positive integer sequences $a_0,a_1,\ldots,a_n$ is {\em unimodality} which means that there is some index $m$ such that
$$
a_0\le a_1\le \ldots\le a_m \ge a_{m+1}\ge \ldots \ge a_n.
$$
So, one may ask if the sequences of coefficients of the polynomials $S[n, k]$, 
$S_B[n, k]$, $c[n, k]$, or $c_B[n, k]$ are unimodal or even log-concave for each particular choice of $n, k$. Using a brute-force computation for $n, k \leq 50$, we have the following conjecture.

\begin{conj}
  For each $n, k$, the coefficients of $S[n, k]$ and $c[n, k]$ are log-concave and positive, hence unimodal.
\end{conj}

In type $B$, the coefficients are not necessarily even unimodal. The first counterexamples are as follows:
\begin{align*}
    S_B[6, 4]
      &= 15 + 24 q + 34 q^2 + 38 q^3 + 43 q^4 + 42 q^5 + 43 q^6 + 38 q^7 + 
 35 q^8 \\
      &+ 26 q^9 + 20 q^{10} + 14 q^{11} + 10 q^{12} + 6 q^{13} + 4 q^{14} + 
 2 q^{15} + q^{16} \\
    c_B[7, 5]
      &= 21 + 36 q + 51 q^2 + 60 q^3 + 70 q^4 + 74 q^5 + 79 q^6 + 78 q^7 + 79 q^8 \\
      &+ 74 q^9 + 71 q^{10} + 62 q^{11} + 56 q^{12} + 44 q^{13} + 35 q^{14} + 26 q^{15} + 20 q^{16} \\
      &+ 14 q^{17} + 10 q^{18} + 6 q^{19} + 4 q^{20} + 2 q^{21} + q^{22}
\end{align*}

A sequence $a_0, a_1, a_2, \ldots$ is \textit{parity-unimodal} if $a_0, a_2, a_4, \ldots$ and $a_1, a_3, a_5, \ldots$ are each unimodal. See the article of Billey, Konvalinka, and Swanson~\cite[\S9]{BKS:tpf} for additional instances of this notion. Inspired by this definition, we say that such a sequence is \textit{parity log-concave} if $a_0, a_2, a_4, \ldots$ and $a_1, a_3, a_5, \ldots$ are each log-concave and similarly for parity unimodal.

\begin{conj}
  For each $n, k$, the coefficients of $S_B[n, k]$ and $c_B[n, k]$ are parity log-concave and positive, hence parity unimodal.
\end{conj}

Another common property of a sequence $a_0, a_1, a_2, \ldots, a_n$ is that it is {\em symmetric} meaning that $a_k=a_{n-k}$ for all $k\in[n]$.  From the examples above, it is clear that the Stirling polynomials do not have symmetric coefficients, but there is a related condition that they seem to enjoy.
A sequence  is {\em bottom heavy} if $a_k \geq a_{n-k}$ for $k < n/2$.  See the article of McConville, Sagan, and Smyth~\cite{MSS:ruc} as well as the references therein for more about the bottom heavy condition.

\begin{conj}
  For each $n, k$, the coefficients of $S[n, k]$, $c[n, k]$, $S_B[n, k]$, and $c_B[n, k]$ are all bottom heavy.
\end{conj}

Both of the previous conjectures have been verified by computer  for $n, k \leq 50$.
A property which implies both bottom heaviness and unimodality is being {\em bottom interlacing} which means that
$$
a_n\le a_0  \le a_{n-1} \le a_1 \le \ldots \le a_{\fl{n/2}}
$$
where $\fl{\cdot}$ is the floor function.
Again, see~\cite{MSS:ruc} for more information about this property.
The coefficients of $c[n, k]$, $S[n, k]$, $c_B[n, k]$, and  $S_B[n, k]$ are not in general bottom interlacing. For example, 
$$
S[4, 3] = 3+2q+q^2
$$
does not have this property.

\subsection{Asymptotics}

A sequence of real-valued random variables $\mathcal{X}_1, \mathcal{X}_2, \ldots$ is \textit{asymptotically normal} if the sequence of standardized random variables $\mathcal{X}_1^*, \mathcal{X}_2^*, \ldots$ converges in distribution to the standard normal distribution $\mathcal{N}(0, 1)$. More explicitly, this means that for all $t \in \mathbb{R}$,
  \[ \lim_{n \to \infty} \mathbb{P}\left[\frac{\mathcal{X}_n - \mu_n}{\sigma_n} \leq t\right] = \int_{-\infty}^t \frac{1}{\sqrt{2\pi}} \exp(-x^2/2)\,dx. \]
In order for $\mathcal{X}_i^*$ to be well-defined, we must assume $\mathcal{X}_i$ does not have the degenerate distribution with variance $0$ supported at a single point. This is only a minor inconvenience, since degenerate distributions are normal with variance $0$ and are ``morally'' if not technically asymptotically normal.

We say that a sequence of non-zero polynomials $P_1(q), P_2(q), \ldots$ with non-negative real coefficients $P_n(q) = \sum_{k \geq 0} a_n(k) q^k$ is {\em asymptotically normal} if the sequence of random variables defined by $\mathbb{P}[\mathcal{X}_n = k] = a_n(k)/P_n(1)$ is asymptotically normal. Informally, plotting the list of coefficients of $P_n(q)$ for large $n$ must give a bell-shaped curve. See e.g.~\cite[\S2.4, \S4.1]{BKS:tpf} for further discussion and an example.

The polynomials $p_n(q)=\sum_{k\ge0} c(n,k) q^k$ and
$P_n(q)=\sum_{k\ge0} S(n,k) q^k$  were shown to be asymptotically normal by Bender~\cite{ben:cll} and by Harper~\cite{har:sha}, respectively. It is natural to ask for asymptotic estimates of the coefficients of $q$-Stirling numbers analogous to these classic results, which we have been unable to find in the existing literature. We content ourselves with the following simple case whose proof is an easy application of Bender's well-known result involving bivariate generating functions.

\begin{thm}\label{thm:Snk_an}
  Fix $k \in \mathbb{Z}_{\geq 0}$. The coefficients of  $S[n, k]$ when $k \geq 2$, and of  $S_B[n, k]$ when $k \geq 1$ are asymptotically normal as $n \to \infty$.
\end{thm}

\begin{proof}
As we saw in \Cref{ogf},  the generating functions are
$$
\sum_{n \geq 0} S[n, k] x^n = 
\frac{x^k}{(1 - [1] x)(1 - [2] x)\cdots (1 - [k] x)} 
$$
and
$$ \sum_{n \geq 0} S_B[n, k] x^n
= \frac{x^k}{(1 - [1] x)(1 - [3] x)\cdots (1 - [2k+1] x)}. 
$$
  Each is of the form $g(x, q)/P(x, q)$ where 
  \ben
  \item $P(x, q)$ is a polynomial in $x$ with coefficients continuous in $q$, 
  \item $P(x, 1)$ has a simple root at $r = 1/k$ or $r = 1/(2k+1)$ with all other roots having larger absolute value, 
  \item $g(x, q) = x^k$ is entire, and 
  \item $g(r, 1) \neq 0$. 
  \een
  The distributions are non-degenerate for $n$ large when $k \geq 2$ in type $A$ and when $k \geq 1$ in type $B$. The result follows from \cite[Ex.~3.1, p.95]{ben:cll} (which inadvertently neglects the $\sigma \neq 0$ condition).
\end{proof}

A direct analogue of \Cref{thm:Snk_an} for the $q$-Stirling numbers of the first kind would require explicit expressions for $\sum_{n \geq 0} c[n, k] x^n/n!$ or $\sum_{n \geq 0} c_B[n, k] x^n/n!$, which we do not have. 
In a complementary direction, the distributions of the coefficients of $S[n, n-k]$, $S_B[n, n-k]$, $c[n, n-k]$, and  $c_B[n, n-k]$ for fixed $k$ as $n \to \infty$ appear to be non-normal, e.g. when $k=1$ the limiting standardized disribution in each case is the triangle with density $(2 \sqrt{2} - x)/9$ for $-\sqrt{2} \leq x \leq 2 \sqrt{2}$.

Based on computational data, the coefficients of the $q$-Stirling numbers of both kinds and in both types   all appear to be ``generically asymptotically normal.'' The following conjecture is one way to make this intuition precise.

\begin{conj}
  Suppose $n, k \to \infty$ in such a way that $k/n \to \alpha$ for some $0 < \alpha < 1$. Then the coefficients of $S[n, k]$, 
  $S_B[n, k]$, $c[n, k]$, and $c_B[n, k]$ are all asymptotically normal.
\end{conj}

More strongly, the coefficients of $S[n, k], S_B[n, k], c[n, k], c_B[n, k]$ appear to tend towards a ``limit shape.'' See \Cref{S_ns} for an example with $S[n, k]$. 
Note that the slices for fixed $k$ appear to be parabolic near their maximum, consistent with asymptotic normality, at least for $k$ not close to $0$ or $n$.
It would be interesting to find the limit shape precisely in each case. It would also be very interesting to develop tools for proving limit shapes of the coefficients arising from recursions similar to \eqref{eq:Sq_rec}, perhaps by exploiting their $q$-difference equations.

\begin{figure}
  \centering
  \includegraphics[scale=0.5]{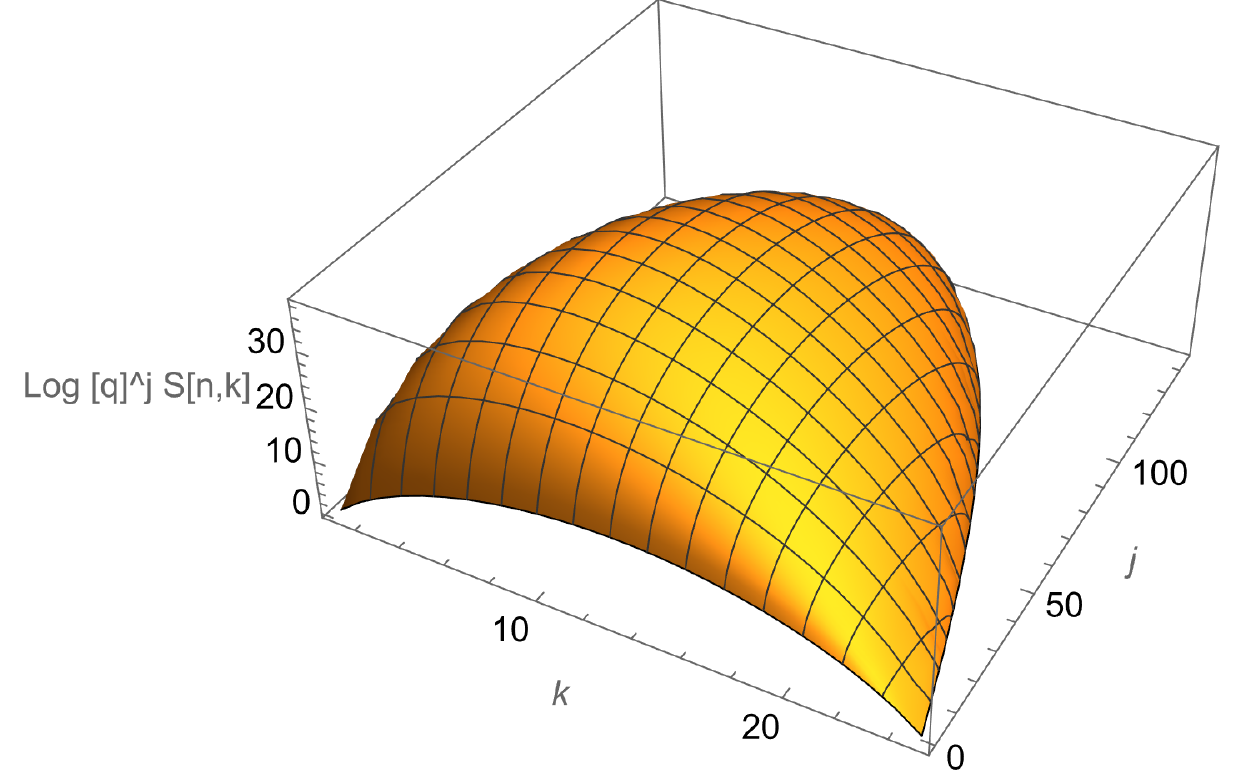}
  \hspace{1cm}
  \includegraphics[scale=0.5]{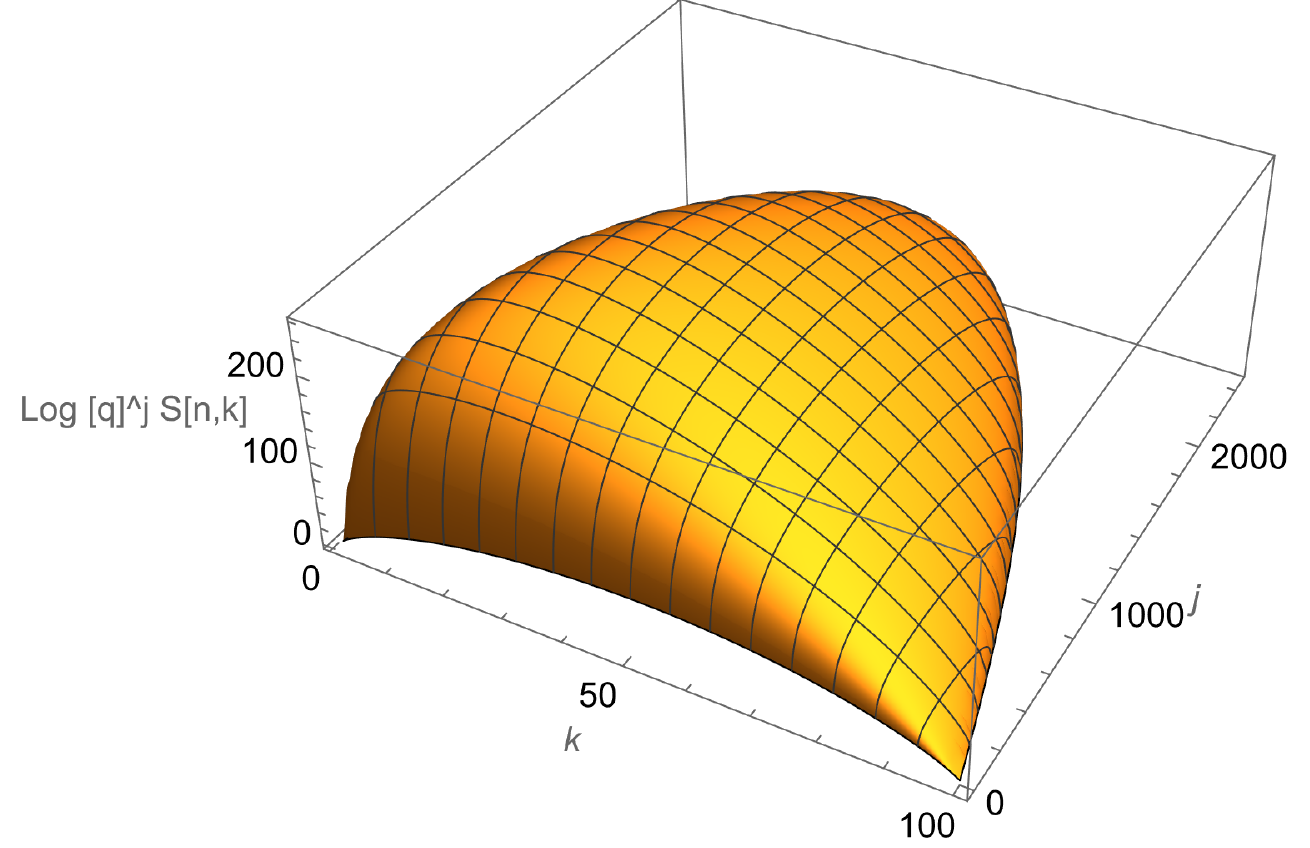}
  \caption{Plots of $\log([q]^j S[n, k])$ as a function of $k$ and $j$ for $n=25$ and $n=100$ }
  \label{S_ns}
\end{figure}

\subsection{Graded Euler characteristics}

By \cite{SW:hdf}, the alternating sums 
$$
\left(\sum_{k=0}^n (-1)^{n-k} q^{m(n-k)} S^o[n, k]\right) - 1
\qquad\text{and}\qquad
\left(\sum_{k=0}^n (-1)^{n-k} q^{m(n-k)} S^o_B[n, k]\right) - 1
$$
are conjecturally graded Euler characteristics of chain complexes obtained by applying generalized exterior derivatives to super coinvariant algebras in types $A$ and $B$, respectively. When $m=1$, the complexes in question use classical exterior differentiation and are an algebraic analogue of the de Rham complex. One of the main results of \cite{SW:hdf} shows that the complex is exact in this case, and correspondingly the $m=1$ alternating sums are $0$. 

The following specific alternating sum appears to exhibit significant structure.

\begin{conj}
  The polynomial
  \begin{equation}\label{eq:chi_2_A}
   \left( \sum_{k=0}^n (-1)^{n-k} q^{2(n-k)} S^o[n, k]\right) - 1
  \end{equation}
  is palindromic ignoring signs, with the same number of positive and negative coefficients, and where the lower-degree half of nonzero coefficients are positive and the rest are negative.
\end{conj}

The coefficients of \eqref{eq:chi_2_A} in fact appear to be slight variations on \cite[A050176]{oeis}. See \Cref{tab:chi_2_A} for examples.

\begin{table}
    \centering
    \begin{tabular}{c|c}
        $\sum_{k=0}^n (-1)^{n-k} q^{2(n-k)} S^o[n, k] - 1$ & $(n+1)$st row of A050176 \\
        \midrule
        $-6q^7 - 14q^6 - 14q^5 + 14q^3 + 14q^2 + 6q + 1$ & $1, 6, 14, 14, 14, 14, 6, 1$ \\
        $-7q^8 - 20q^7 - 28q^6 - 14q^5 + 14q^4 + 28q^3 + 20q^2 + 7q + 1$ & $1, 7, 20, 28, 14, 28, 20, 7, 1$
    \end{tabular}
    \caption{The $n=7$ and $n=8$ cases of \eqref{eq:chi_2_A} and the $8$th and $9$th rows of \cite[A050176]{oeis}.}
    \label{tab:chi_2_A}
\end{table}



\nocite{*}
\bibliographystyle{alpha}

\end{document}